\documentclass[10pt,reqno]{amsart}

\usepackage{amsfonts,latexsym,url,amsmath,amssymb,comment,enumitem}
\usepackage{hyperref}
\excludecomment{commentforus}

\usepackage{url,centernot}
\usepackage{stmaryrd}
\usepackage{amsmath,amsthm,amssymb}
\usepackage[cal=dutchcal]{mathalfa}
\usepackage{gensymb}
\usepackage[square,sort&compress,comma,numbers]{natbib}
\usepackage{subfig}
\usepackage{graphicx}
\usepackage[margin=3cm]{geometry}
\usepackage{tikz}
\usepackage{caption}
\usepackage{mathrsfs}
\usetikzlibrary{shapes,calc}
\usepackage{verbatim}
\chardef\bslash=`\\ 

\newtheorem{thm}{Theorem}[section]
\newtheorem{cor}[thm]{Corollary}
\newtheorem{lem}[thm]{Lemma}

\newtheorem{rem}[thm]{Remark}

\theoremstyle{definition}
\newtheorem{defn}{Definition}[section]

\numberwithin{equation}{section}

\newcommand{\bxi}{\boldsymbol{\xi}}

\newcommand{\bz}{\boldsymbol{z}}
\newcommand{\bw}{\boldsymbol{w}}

\newcommand{{\bH}}{{\bf H}}

\newcommand{\bV}{{\bf V}}

\newcommand{\cA}{\mathcal A}
\newcommand{\cC}{\mathcal C}
\newcommand{\cB}{\mathcal B}

\newcommand{\cT}{\mathcal T}
\newcommand{\cL}{\mathcal L}

\newcommand{\cP}{\mathcal P}
\newcommand{\cof}{{\rm cof}}

\newcommand{\map}{\longrightarrow}

\newcommand{\half}{\frac{1}{2}}
\newcommand{\fl}{\;\; \forall\,}

\newcommand{\hto}{H^2_0(\Omega)}

\newcommand{\integ}{\int_\Omega}
\newcommand{\dx}{\;\textit{dx}}

\newcommand{\trinl}{\ensuremath{\left| \! \left| \! \left|}}
\newcommand{\trinr}{\ensuremath{\right| \! \right| \! \right|}}

\newcommand{\vket}{von K\'{a}rm\'{a}n equations }
\newcommand{\BFS}{Bogner-Fox-Schmit }

\newcommand{\Poincare}{Poincar\'{e} }

\usepackage[normalem]{ulem}
\definecolor{violet}{rgb}{0.580,0.,0.827}

\begin{document}
\title[Error estimates for optimal control of von K\'{a}rm\'{a}n equations]{Error estimates for the numerical approximation of a  distributed optimal control problem governed by the  von K\'{a}rm\'{a}n equations}
\date{\today}
\maketitle
\begin{center}
\author{        
Gouranga Mallik
\footnote{Department of Mathematics, Indian Institute of Technology Bombay, Powai, Mumbai 400076, India. Email. gouranga@math.iitb.ac.in}
$\boldsymbol{\cdot}$ Neela Nataraj
\footnote{Department of Mathematics, Indian Institute of Technology Bombay, Powai, Mumbai 400076, India. Email. neela@math.iitb.ac.in}
$\boldsymbol{\cdot}$ Jean-Pierre Raymond
 \footnote{
  Institut Math\'{e}matiques de Toulouse, UMR CNRS 5219, Universit\'{e} Paul Sabatier Toulouse III, 31062 Toulouse Cedex 9, France. Email. raymond@math.univ-toulouse.fr}
}
\end{center}
\begin{abstract} 
 In this paper, we discuss the numerical approximation of a distributed optimal control problem governed by the  von K\'{a}rm\'{a}n equations, defined in polygonal domains with point-wise control constraints. Conforming finite elements are employed to discretize the state and adjoint variables. The control is discretized using piece-wise constant approximations.  {\it  A~priori} error estimates are derived for the state, adjoint and control variables under minimal regularity assumptions on the exact solution. Numerical results that justify the theoretical results are presented.
\end{abstract}

\bigskip
\noindent{\bf Keywords:} {von K\'{a}rm\'{a}n equations, distributed control,  plate bending, non-linear, conforming  finite element methods, error estimates}

\bigskip

{\bf AMS subject classifications.} \subjclass{65N30, 65N15, 49M05, 49M25}

\section{Introduction}
Consider the  distributed control problem governed by the von K\'{a}rm\'{a}n equations defined by: 
\begin{subequations}
\begin{align}
&  \min_{u \in U_{ad}} J(\Psi, u) \,\,\, \textrm{ subject to } \label{cost}\\
&  \Delta^2  \psi_1 =[\psi_1 ,\psi_2]+f +{\mathcal C}u \,\, \quad \mbox{ in } \Omega, \label{state1} \\
&  \Delta^2  \psi_2 =- \half[\psi_1 ,\psi_1] \,\, \quad \mbox{ in } \Omega,  \label{state2} \\
& \psi_1=0,\,\frac{\partial \psi_1}{\partial \nu}=0\text{ and } \psi_2=0,\,\frac{\partial  \psi_2}{\partial \nu} =0
\text{  on }\partial\Omega, \label{state3}
\end{align}
\end{subequations}
where $\Psi=(\psi_1,\psi_2)$ and the components $\psi_1$ and $\psi_2$ denote the displacement and Airy-stress respectively,  $\Delta^2\varphi:=\varphi_{xxxx}+2\varphi_{xxyy}+\varphi_{yyyy},$ the von K\'{a}rm\'{a}n bracket $ [\eta,\chi]:=\eta_{xx}\chi_{yy}+\eta_{yy}\chi_{xx}-2\eta_{xy}\chi_{xy}$ and
 $\nu$ is the unit outward normal to the boundary $\partial\Omega$ of the polygonal domain $\Omega \subset {\mathbb R}^2$. The load function $f \in H^{-1}(\Omega)$,  
  ${\mathcal C} \in {\mathcal L}(L^2(\omega), L^2(\Omega))$ is a localization operator defined by ${\mathcal C}u(x) = u(x) \chi_{\omega}(x)$, where $\chi_{\omega}$ is the characteristic function of $\omega \subset L^2(\Omega)$,   the cost functional $J(\Psi, u)$ is defined by
\begin{equation}
J(\Psi, u) :=\frac{1}{2}\|\Psi- \Psi_d \|_{L^2(\Omega)}^2  + \frac{\alpha}{2} \int_{\omega} |{u}|^2 \: dx,
\end{equation}
with a fixed regularization parameter $\alpha>0$, $\Psi_d=(\psi_{1d}, \psi_{2d})$ is the given observation for $\Psi$ and $U_{ad} \subset L^2(\omega)$ is a non-empty, convex and bounded admissible space of controls defined by
\begin{equation} \label{controlspace}
U_{ad} = \{   u \in L^2(\omega) : u_a \le  u(x) \le u_b  \; \; \mbox{ for almost every }  x \mbox { in }  \omega \},
\end{equation}
$u_a, \: u_b \in {\mathbb R}, \; u_a \le u_b $ are given.

\medskip 
Let us first discuss the results available for the state equations, for a given $u \in L^2(\omega)$. For results regarding the existence of solutions, regularity and bifurcation phenomena of the \vket we refer to \cite{CiarletPlates, Knightly, Fife, Berger,BergerFife, BlumRannacher} and the references therein. It is well known \cite{BlumRannacher} that in a polygonal domain $\Omega$, for given $f\in H^{-1}(\Omega)$, the solution of the biharmonic problem belongs to $\hto\cap H^{2+\gamma}(\Omega)$, where $\gamma\in (\half,1]$, referred to as the index of elliptic regularity, is determined by the interior angles of $\Omega$. Note that when $\Omega$ is convex, $\gamma=1$ and the solution belongs to $\hto\cap H^3(\Omega)$. 
It is also stated in \cite{BlumRannacher} that similar regularity results hold true for von K\'{a}rm\'{a}n equations, that is, the solutions $\psi_1, \psi_2$ belong to $\hto\cap H^{2+\gamma}(\Omega)$.  However, we give the details of the arguments of this proof in this paper.

\medskip
\noindent 
Due to the importance of the problem in application areas, several numerical approaches have  been attempted in the past for the state equation.   The major challenges of the problem from the numerical analysis point of view  are the non-linearity and the higher order nature of the equations. In \cite{Brezzi, Miyoshi, Quarteroni},  the authors consider the state equation with {\it homogeneous} boundary conditions and study the approximation and error bounds for {\it nonsingular} solutions. The convergence analysis has been studied using conforming finite element methods in \cite{Brezzi, ng1}, nonconforming finite element methods in \cite{ng2}, $C^0$ interior penalty method \cite{brennernew}, the Hellan-Hermann-Miyoshi mixed finite element method in \cite{Miyoshi, Reinhart} and a stress-hybrid method in \cite{Quarteroni}, respectively.   

\medskip
For the numerical approximation of  the distributed control problem defined in \eqref{cost}-\eqref{state3}, not many results are available in literature. The paper \cite{hou} discusses conforming finite element approximation of the problem defined in convex domains {\it without} control constraints and when the control is defined over the whole domain $\Omega$, whereas \cite{GunzburgerHou1996}  discusses an abstract framework for a class of nonlinear optimization problems using a Lagrange multiplier approach. For results on optimal control problems of the steady-state Navier-Stokes equations, with and without control constraints,  many references are available, see for example, \cite{cmj}, \cite{GHS91},  \cite{GHS91A} to mention a few. Employing a post processing of the discretized control $u$, \cite{MeyerRosh04,GudiNatarajPorwal14} establish a super convergence result for the control variable for the second-order and fourth-order linear elliptic problems.


\medskip \noindent  In this paper, we discuss the existence of solutions of conforming finite element approximations of  {\it local nonsingular} solutions of the control problem and establish {\it a priori} error estimates. The contributions of this paper are
\begin{itemize}
\item[(i)] error estimates for the state and adjoint variables in the energy norm  and that for the control variable  in the $L^2$ norm,  under realistic regularity assumptions for the exact solution of the problem defined in polygonal domains,
\item[(ii)] a  guaranteed quadratic convergence result in convex domains for a post processed control \cite{MeyerRosh04} constructed by projecting the discrete adjoint state into the admissible set of controls,
\item[(iii)]  error estimates for state and adjoint variables in lower order  $H^1$ and $L^2$ norms,
\item[(iv)] numerical results that illustrate all the theoretical estimates.
\end{itemize}

\medskip \noindent
Throughout the paper, standard notations on Lebesgue and
Sobolev spaces and their norms are employed.
The standard semi-norm and norm on $H^{s}(\Omega)$ (resp. $W^{s,p} (\Omega)$) for $s>0$ are denoted by $|\cdot|_{s}$ and $\|\cdot\|_{s}$ (resp. $|\cdot|_{s,p}$ and $\|\cdot\|_{s,p}$ ). The standard $L^2$ inner product is denoted by $(\cdot, \cdot)$.  We use the notation ${\bH}^s(\Omega)$ (resp. ${\bf L}^p(\Omega)$) to denote the product space  
 $H^{s}(\Omega) \times H^s(\Omega)$ (resp. $L^p(\Omega) \times L^p(\Omega)$). The notation $a\lesssim b$ means there exists a generic mesh independent constant $C$ such that $a\leq Cb$. The positive constants $C$ appearing in the inequalities denote generic constants which do not depend on the mesh-size. 

\medskip \noindent The rest of the paper is organized as follows. The weak formulation and some auxiliary results needed for the analysis are described in Section 2. The state and adjoint variables are discretized and some intermediate discrete problems along with error estimates are established in Section 3. In Section 4, the discretization of the control variable is described and the existence and convergence results for the fully discrete problem are established. This is followed by derivation of the error estimates for the state, adjoint and control variables in Section 5.  The post processing of control for improved error estimates and lower order estimates for state and adjoint variables are also derived. Section 6 deals with the results of the numerical experiments.  The discrete optimization problem is solved using the primal-dual active set strategy, see for example \cite{fredi2010}. The state and adjoint variables are discretized using \BFS finite elements and the control variable is discretized using piecewise constant functions. The post-processed control is also computed.

The analysis that we do in Sections 2 and 3 and several results stated there are very close to what is done in \cite{cmj} for the Navier-Stokes system. However, many of the proofs in our paper are based on results specific to the von K\'{a}rm\'{a}n equations. This is why we  have to adapt several results from \cite{cmj} to our setting. 

\section{Weak formulation and Auxiliary results}
\label{sec:weakformulation}
In this section, we state the weak formulation corresponding to \eqref{cost}-\eqref{state3} in the first subsection and present some auxiliary results in the second subsection. This will be followed by the derivation of first order and second order optimality conditions for the control problem in the third subsection.

\subsection{Weak formulation}
The weak formulation corresponding to \eqref{cost}-\eqref{state3} reads:
\begin{subequations}\label{wform}
  \begin{align} 
  &  \min_{(\Psi, {u}) \in  \bV \times  U_{ad}} J(\Psi, u) \,\,\, \textrm{ subject to } \label{cost1}\\
   &  A(\Psi,\Phi)+B(\Psi,\Psi,\Phi)=(F + {\mathbf C}{\bf u}, \Phi), \label{wv}
  \end{align}
\end{subequations}
for all $\Phi \in \bV$, where  $ \bV:=V \times V$ with $V :=H^2_0(\Omega)$.  For all $ {\boldsymbol \xi}=({\boldsymbol \xi}_{1},{\boldsymbol \xi}_{2}),{\boldsymbol \lambda}=({\boldsymbol \lambda}_{1},{\boldsymbol \lambda}_{2})$, $ \Phi=(\varphi_{1},\varphi_{2})\in \bV$,
   \begin{align*}
    &A({\boldsymbol \lambda},\Phi):=a({\boldsymbol \lambda}_1,\varphi_1)+a({\boldsymbol \lambda}_2,\varphi_2), \\
    &B({\boldsymbol \xi},{\boldsymbol \lambda},\Phi):=b({\boldsymbol \xi}_{1},{\boldsymbol \lambda}_{2},\varphi_{1})+b({\boldsymbol \xi}_{2},{\boldsymbol \lambda}_{1},\varphi_{1})-b({\boldsymbol \xi}_{1},{\boldsymbol \lambda}_{1},\varphi_{2}), \; \;  \\
    & F= \left( \begin{array}{c}
f  \\
0  \end{array} \right) , \; {\mathbf C}{\bf u} =
\left( \begin{array}{c}
{\mathcal C}u \\
0  \end{array} \right), \;
 {\bf u} =\left( \begin{array}{c}
u \\
0  \end{array} \right)
   \mbox{ and }
    (F + {\mathbf C}{\bf u} ,\Phi):=(f+{\mathcal C}u,\varphi_1), 
   \end{align*} 
 \noindent and for all $\eta,\chi,\varphi\in V$, 
  \begin{align*}
   &a(\eta,\chi):=\integ D^2 \eta:D^2\chi\dx, \; \;  b(\eta,\chi,\varphi):=\half\integ \cof(D^2\eta)D\chi\cdot D\varphi\dx. \; \;       \; \;   
  \end{align*}
  
\begin{rem}
Note that 
\begin{equation}
\int_{\Omega} [\eta, \chi] \varphi \dx = - \int_{\Omega}  \cof(D^2\eta)D\chi\cdot D\varphi\dx, \; \; \; \forall \eta, \chi, \varphi \in H^2_0(\Omega).
\end{equation}
The form $b(\cdot,\cdot,\cdot)$ is derived using the divergence-free rows property~\cite{Evans}. Since the Hessian matrix $D^2\eta$ is symmetric, $\cof(D^2\eta)$ is symmetric. Consequently, $b(\cdot,\cdot,\cdot)$ is symmetric with respect to the  second and third variables, that is,   $b(\eta,{\boldsymbol \xi},\varphi)=b(\eta,\varphi,{\boldsymbol \xi})$. Moreover,  since $[\cdot,\cdot]$ is symmetric, $b(\cdot,\cdot,\cdot)$ is symmetric with respect to all the variables.   Also, $B(\cdot,\cdot,\cdot)$ is symmetric in the first and second variables due to the symmetry of $b(\cdot,\cdot,\cdot)$.    
\end{rem}

\medskip
\noindent The Sobolev space $\bV$ is equipped with the norm defined by
\begin{align*}
 \trinl\Phi\trinr_2&:=(|\varphi_1|_{2,\Omega}^2+|\varphi_2|_{2,\Omega}^2)^{\frac{1}{2}} \; \;  \; \fl\Phi=(\varphi_1,\varphi_2)\in \bV,
\end{align*}
 where $\displaystyle |\varphi|^{2}_{2, \Omega}=\int_\Omega D^2\varphi:D^2\varphi\dx$, for all $\varphi\in V$.  
 
 In the sequel, the product norm defined on ${\bH}^s(\Omega)$ and ${\bf L}^2(\Omega)$ are denoted by $\trinl\cdot\trinr_s$ and $\trinl\cdot\trinr$, respectively.  

\medskip
The following properties of boundedness and coercivity of $A(\cdot, \cdot)$ and boundedness of $B(\cdot, \cdot, \cdot)$  hold true: 
$\forall {\boldsymbol \xi},{\boldsymbol \lambda},\Phi\in \bV$,
\begin{eqnarray}
 && |{A}({\boldsymbol \xi}, \Phi)| \leq \trinl{\boldsymbol \xi}\trinr_2 \: \trinl\Phi\trinr_2,\label{boundA}\\
& & |{A}({\boldsymbol \xi},{\boldsymbol \xi})| \geq \trinl   {\boldsymbol \xi} \trinr_2^2, \label{coercivity}\\
& {\rm and } & |B({\boldsymbol \xi}, {\boldsymbol {\boldsymbol \lambda}}, \Phi)|  \leq  C_b \trinl{\boldsymbol \xi}\trinr_2 \: \trinl {\boldsymbol {\boldsymbol \lambda}} \trinr_2 \: \trinl\Phi\trinr_2. \label{boundB1}
\end{eqnarray}

With the definition of $B(\cdot,\cdot,\cdot)$, the symmetry of $b(\cdot,\cdot,\cdot)$ and the Sobolev imbedding, it yields \cite{ng1}
\begin{align}\label{boundB2}
|B(\Xi,\Theta,\Phi)|\lesssim
\begin{cases}
	\trinl \Xi\trinr_{2+\gamma} \trinl \Theta\trinr_ {2}\trinl \Phi\trinr_{1} \fl \Xi\in {\bH}^{2+\gamma}(\Omega)\text{ and } \Theta,\Phi\in \bV,\\
	\trinl \Xi\trinr_{2+\gamma} \trinl \Theta\trinr_ {1}\trinl \Phi\trinr_{2}\fl \Xi\in {\bH}^{2+\gamma}(\Omega)\text{ and } \Theta,\Phi\in \bV,\\
	\trinl \Xi\trinr_{1} \trinl \Theta\trinr_ {2+\gamma}\trinl \Phi\trinr_{2}\fl \Xi\in \bV, \Theta\in {\bH}^{2+\gamma}(\Omega) \text{ and }\Phi\in \bV,
\end{cases}
\end{align}
where $\gamma \in(\half,1]$ denotes the elliptic regularity index. The above estimates are also valid if $\gamma$ is replaced by any 
$\gamma_0  \in (1/2,\gamma)$, that is 
\begin{equation}
  |B(\Xi,\Theta,\Phi)| \leq C_\gamma\trinl \Xi\trinr_{2+\gamma_0} \trinl \Theta\trinr_ {2}\trinl \Phi\trinr_{1}, \label{boundB3}
\end{equation}
for all $\gamma_0 \in (1/2,\gamma)$.

\noindent We now prove another boundedness result which will be also needed later.

\begin{lem}
For  $\Xi, \Theta,\Phi\in \bV$, there holds 
\begin{equation}
  |B(\Xi,\Theta,\Phi)| \leq C_{\epsilon}\trinl \Xi\trinr_{2} \trinl \Theta\trinr_ {2}\trinl \Phi\trinr_{1+\epsilon}, \quad 0 < \epsilon < 1/2. \label{boundB_new}
\end{equation}
\end{lem}
\begin{proof} It is enough to prove that 
$$\int_\Omega \cof(D^2\xi)D\theta\cdot D\varphi\dx \le C_{\epsilon} \trinl \xi \trinr_2 
\trinl \theta \trinr_2 \trinl \varphi \trinr_{1+\epsilon} \fl \xi,\theta,\varphi\in V.  $$
For $0<\epsilon <1/2$, we have 
\begin{eqnarray*}
\int_\Omega \cof(D^2\xi)D\theta\cdot D\varphi\dx &\le & 
\trinl\cof(D^2\xi)\trinr \trinl D\theta \trinr_{ L^{2/\epsilon}(\Omega)} 
\trinl D\varphi \trinr_{L^{2/1-\epsilon}(\Omega)}  \\
& \le & C_{\epsilon} \trinl\xi \trinr_2 \trinl \theta\trinr_2 \trinl\varphi\trinr_{1+\epsilon}.
\end{eqnarray*}
The last inequality follows from the imbeddings 
$$H^1(\Omega) \hookrightarrow { L}^{2/\epsilon}(\Omega), \; \; H^{\epsilon}(\Omega) \hookrightarrow { L}^{2/(1-\epsilon)}(\Omega)  \mbox{ for } 0 < \epsilon<1/2. $$
The proof is complete.
\end{proof}

  \subsection{Some auxiliary results}
\noindent Define the operator ${\mathcal A} \in {\mathcal L} (\bV, \bV')$ by 
$$
\Big\langle {\mathcal A} \Psi ,\Phi   \Big\rangle_{\bV',\bV } = A(\Psi, \Phi)\qquad \forall \Psi, \Phi \in \bV,
$$
and the nonlinear operator ${\mathcal B}$ from $\bV$ to
$\bV'$  by
$$
\Big\langle {\mathcal B} (\Psi), \Phi \Big\rangle_{\bV',\bV }=B(\Psi, \Psi, \Phi) \qquad \forall \Psi, \Phi \in \bV.
$$
For simplicity of notation, the duality pairing $\Big\langle \cdot,
\cdot \Big\rangle_{\bV',\bV }$ is denoted by $\Big\langle \cdot,\cdot \Big\rangle$.

In the sequel, the weak formulation \eqref{wv} will also be written as 
\begin{equation}
\Psi \in \bV, \; \; {\mathcal A} \Psi + {\mathcal B}(\Psi) = {\bf F} +{\mathbf C}{\bf u}  \mbox{ in } \bV'.  \label{of}
\end{equation}
Note that the nonlinear operator ${\mathcal B}(\Psi)$ can also be expressed in the form
 $${\mathcal B}(\Psi):=\left( \begin{array}{c}
-[\psi_1,\psi_2]  \\
 {\small \frac{1}{2}}[\psi_1,\psi_1]  \end{array} \right).$$
It is easy to verify that, for all $\Psi \in \bV$, the operator ${\mathcal B}'(\Psi) \in {\mathcal L}(\bV, \bV')$ and its adjoint operator ${\mathcal B}'(\Psi)^* \in {\mathcal L}(\bV, \bV')$ satisfy 
\begin{align}
 \Big \langle {\mathcal B}'(\Psi){ {\boldsymbol \xi}}, \Phi \Big \rangle &=
B(\Psi, { {\boldsymbol \xi}}, \Phi) + B( { {\boldsymbol \xi}}, \Psi, \Phi) \; \; \fl  {{ {\boldsymbol \xi}}}, \Phi \in \bV,\\
\Big \langle {\mathcal B}'(\Psi)^*{ {\boldsymbol \xi}}, \Phi \Big \rangle&=
B(\Psi, \Phi, {\boldsymbol \xi}) + B(\Phi, \Psi,  {{\boldsymbol \xi}}) \; \; \fl  { { {\boldsymbol \xi}}} , \Phi \in \bV. \label{opad}
\end{align}
Moreover, ${\mathcal B}'' \in {\mathcal L}(\bV \times \bV, \bV')$ satisfies
\begin{align}
\Big \langle {\mathcal B}''(\Psi,{ {\boldsymbol \xi}}), \Phi \Big \rangle &=
B(\Psi, { {\boldsymbol \xi}}, \Phi) + B( { {\boldsymbol \xi}}, \Psi, \Phi) \; \; \fl  \Psi,  {{  {\boldsymbol \xi}}}, \Phi \in \bV.
\end{align}
\begin{thm}[\it Existence \cite{Knightly,CiarletPlates}] For given $u \in L^2(\omega)$, the problem \eqref{wv} possesses at least one solution. 
\end{thm}  
\smallskip
\noindent A linearization of \eqref{wv} around $\Psi$ in the direction ${\boldsymbol \xi}$ is given by
\begin{equation*}
\boldsymbol{L} {\boldsymbol \xi}:= {\mathcal A}{\boldsymbol \xi}+{\mathcal B}'(\Psi) {\boldsymbol \xi}.
\end{equation*}

\begin{defn}[\it Nonsingular solution] \label{isolated} For a given $u \in L^2(\omega)$,  a solution $\Psi$ of \eqref{wv} is said to be regular if the linearized form is nonsingular.
That is, if $\big \langle \boldsymbol{L}{\boldsymbol \xi},\Phi \big \rangle={ 0}$ for all $\Phi\in \bV$, then ${\boldsymbol \xi}=\boldsymbol{0}$. 
In that case, 
we will also say  that the pair $(\Psi,u)$ is a nonsingular solution of \eqref{state1}-\eqref{state2}. 
\end{defn}
\begin{rem}
{}{The dependence of $\Psi$ with respect to $u$ is made explicit with the notation $\Psi_{u}$ whenever it is necessary to do so.}{}
\end{rem}
{}{
\begin{lem}[$Properties$ $ of$  ${\mathcal A}^{-1}$] \label{poa}
The following properties hold true:
\begin{eqnarray*}
& (i) & {\mathcal A}^{-1} \in  {\mathcal L}(\bV', \bV). \\
& (ii) & {\mathcal A}^{-1} \in {\mathcal L}({\bH}^{-1}(\Omega), {\bH}^{2+ \gamma}(\Omega)), \gamma\in (1/2,1] \text{ is the elliptic regularity index}. \\
& (iii) & {\mathcal A}^{-1} \in  {\mathcal L}({\bH}^{-1-\epsilon}(\Omega), {\bH}^{2+ \gamma(1-\epsilon)}(\Omega)) \mbox{ for all } 0 < \epsilon < 1/2.
\end{eqnarray*}
\end{lem}
\begin{proof} The statement $(i)$ follows from the Lax-Milgram Lemma. The statement  $(ii)$ follows from the regularity result for biharmonic problem (see \cite{BlumRannacher}). Now  (iii) follows from (i) and (ii) by interpolation.
\end{proof}}{}
\noindent In the next lemma, we  obtain  $a~priori$ bounds for the solution $\Psi$ of \eqref{wv}.

\begin{lem}[$An~a~priori~estimate$]\label{ap}
 For $f\in H^{-1}(\Omega)$ and $u \in L^2(\omega)$, the solution $\Psi$ of \eqref{wv} belongs to $ {\bH}^{2 + \gamma} (\Omega)$,  $\gamma \in (1/2,1]$ being the elliptic regularity index, and  satisfies the $a~priori$ bounds
\begin{subequations}
\begin{align}
 \trinl\Psi\trinr_2 &\leq C(\|f\|_{-1} + \|u\|_{L^2(\omega)}), \label{app1} \\ 
 \trinl\Psi\trinr_{2+\gamma}  & \leq C\left( \|f\|^3_{-1}+ \|u\|^3_{L^2(\omega)} +\|f\|^2_{-1}+ \|u\|^2_{L^2(\omega)}  +
 \|f\|_{-1} + \|u\|_{L^2(\omega)} \right). \label{app2}
\end{align}

\end{subequations}
\end{lem} 
\begin{proof} 


\noindent From the scalar form of \eqref{wv}, we obtain, 
\begin{align}
a(\psi_1,\varphi_1)&=\int_{\Omega}[\psi_1, \psi_2]\varphi_1\dx+(f+u,\varphi_1)   \;\; \; \forall \varphi_1 \in V, \label{eq1}\\
a(\psi_2,\varphi_2)&=-\half\int_{\Omega}[\psi_1,\psi_1]\varphi_2\dx \; \;\; \forall \varphi_2 \in V. \label{eq2}
\end{align}
Choose $\varphi_1=\psi_1$ in \eqref{eq1} and $\varphi_2=\psi_2$ in \eqref{eq2}, use the result $\displaystyle \int_{\Omega}[\psi_1,\psi_2]\psi_1\dx = \int_{\Omega}[\psi_1,\psi_1] \psi_2\dx $ and the definition of $a(\cdot,\cdot)$  to obtain

$$ |\psi|_2^2+2|\psi|_2^2\leq \ |f\|_{-1}\|\psi_1\|_1 + \|u \|_{L^2(\omega)} \|\psi_1\|_0. $$

%
\noindent An application of \Poincare inequality leads to  \eqref{app1}. 

\medskip 

It is already proved in \cite{BlumRannacher} that \eqref{wv} admits a solution in ${\bH}^2(\Omega)$. From \eqref{boundB_new}, it follows that 
$$|B(\Psi, \Psi,\Phi) | \le C_{\epsilon} \trinl\Psi\trinr_2^2\trinl \Phi\trinr_{1+\epsilon} \mbox{ for } 0 < \epsilon < 1/2.$$
Thus ${\mathcal B}(\Psi)$ belongs to ${\bH}^{-1-\epsilon}(\Omega)$ and 
$\trinl{\displaystyle {\mathcal B}(\Psi)}\trinr_{-1-\epsilon} \le C_{\epsilon} \trinl \Psi\trinr_2^2$.
From Lemma \ref{poa} (iii), it follows that 
$$\trinl \Psi\trinr_{2 + \gamma(1-\epsilon)} \le C_{\epsilon} \left( \trinl \Psi\trinr^2_2 + \|u\| + \|f \|_{-1}\right). $$
Next using \eqref{boundB3}, we obtain 
$$\|{\mathcal B}(\Psi) \|_{-1} \le C \trinl \Psi \trinr_{2 + \gamma(1-\epsilon)} 
 \trinl \Psi \trinr_{2}.$$
 Combining this estimate with Lemma \ref{poa}(ii), we finally obtain the required result \eqref{app2}.
 \end{proof}
Note that $\Psi \in {\bH}^{2+\gamma}(\Omega)$ is already observed in 
\cite{BlumRannacher}, but the arguments are not completely given there and hence we have given a complete proof for clarity.

\noindent The implicit function theorem yields the following result, see \cite{cmj}.

\begin{thm}\label{th2.5}
Let $(\bar{\Psi}, \bar{u}) \in \bV \times L^2(\omega)$
be a nonsingular solution of
\eqref{wv}. 
Then there exist an open ball
 ${\mathcal O}(\bar { u})$ 
 of $\bar{ u }$ in $L^2(\omega)$, an open ball
 ${\mathcal O}({\bar \Psi})$ of
$\bar \Psi$ in $\bV$, and a mapping $G$ from
${\mathcal O}(\bar u)$ to ${\mathcal O}(\bar\Psi)$ of class $C^\infty$,
such that, for all $u\in {\mathcal O}(\bar u)$, $\Psi_{u}=G({u})$  is the unique solution in ${\mathcal
O}(\bar\Psi)$ to  $\eqref{of}$.{}{Thus, $G'(u)= ({\mathcal A} + {\mathcal B}'(\Psi_{u}))^{-1}$ is uniformly bounded from a smaller ball into a smaller ball.}{}$($These smaller balls are still denoted by $ {\mathcal O}(\bar u)$ and ${\mathcal O}(\bar\Psi)$ for notational simplicity.$)$ 
Moreover,  if
$ G'({u}){v} =: \mathbf{z}_{{v}} \in
\bV$ and 
$ G''({u}) {v}^2 =: \mathbf{w} \in
\bV$, then ${\bz}_{v}$ and $\bw$ satisfy the equations
\begin{eqnarray}
{\mathcal A}\mathbf{z}_{{v}} +
{\mathcal B}'(\Psi_{u})\mathbf{z}_{{v}} =
{\mathbf C}\mathbf{v}\quad \mathrm{in\ } \bV,
\label{E2.7}\\
{\mathcal A}{\mathbf{w}} + {\mathcal B}'(\Psi_{u}){\mathbf{w}} +{\mathcal B}''({\mathbf{z}}_v,{\mathbf{z}}_v)
= 0\quad \mathrm{in\ } \bV,  \label{E2.8}
\end{eqnarray}
and $({\mathcal A} + {\mathcal B}'(\Psi_{u}))$ is an isomorphism from $\bV$
into ${\bV}'$ for all $u\in {\mathcal O}(\bar u)$.
\label{T2.2}

Also, the following holds true:
\begin{eqnarray*}
&&\|{\mathcal A} + {\mathcal B}'(\Psi_{u})\|_{{\mathcal L}(\bV, \bV')} \le C, \quad 
\|({\mathcal A} + {\mathcal B}'(\Psi_{u}))^{-1}\|_{{\mathcal L}(\bV', \bV)} \le C \; \; \forall u \in {\mathcal O}(\bar u), \\
& &  \|\mathbf{z}_{{v}}\|_{2}\le \| G'({u}) \|_{{\mathcal L}(L^2(\omega), H^2(\Omega))}\|v\|_{L^2(\omega)}. 
\end{eqnarray*}
\end{thm}

\begin{lem}[{\it A priori bounds for the linearized problem}]\label{linapriori}
 The solution ${\bf z}_{ v }$ of the linearized problem \eqref{E2.7} belongs to ${\bH}^{2 + \gamma}(\Omega) $,  $\gamma \in (1/2,1] \; \mbox{being the elliptic regularity index}$, and  satisfies the $a~priori$ bound
 $$\trinl  z_v \trinr_{2 + \gamma} \le C\|v\|_{L^2(\omega)}.$$
\end{lem}{
\begin{proof}
From Theorem \ref{th2.5}, we know that there exists $C>0$ such that $\|\mathbf{z}_{{v}}\|_{2}\le C$ for $u \in {\mathcal O}(\bar u)$. Now rewriting \eqref{E2.7} in the form
\begin{equation}\label{lax-milgram}
{\mathcal A}\mathbf{z}_{{v}} =
{\mathbf C}\mathbf{v} -{\mathcal B}'(\Psi_{u})\mathbf{z}_{{v}}, 
\end{equation}
and  using Theorem \ref{th2.5} and \eqref{app2}, we obtain, for $u\in \mathcal{O}(\bar{u})$
\begin{equation*}
\| {\mathcal B}'(\Psi_{u})\mathbf{z}_{{v}}\|_{-1} \;\le C \trinl \Psi_{u} \trinr_{2+\gamma} \trinl z_v \trinr_2 \le C \|v\|_{L^2(\omega)}.
\end{equation*}
Since $A(\cdot,\cdot)$ is bounded and coercive, a use of Lemma~\ref{poa}(ii) and the above result  in \eqref{lax-milgram} leads
to the required regularity result \cite{BlumRannacher}.
\end{proof}
}



\noindent The next lemma is an easy consequence of the {\it a priori} bounds in Lemma \ref{ap}.
\begin{lem}\label{L2.5} 
Let $({\bar \Psi}, {\bar u})$ be a nonsingular solution of \eqref{wv}, as defined in Theorem \ref{th2.5}.  Let $({u_k})_k$
be a sequence in ${\mathcal O}(\bar u)$ weakly converging to
$\bar u$ in $L^2(\omega)$. Let $\Psi_{u_k}$ be the
solution to equation \eqref{wv} in ${\mathcal O}(\bar\Psi)$
corresponding to $u_k$.  Then,  $(\Psi_{u_k})_k$ converges to $\bar \Psi$ in
$\bV$.
\end{lem}
\subsection{Optimality Conditions}
In this subsection, we discuss the first order and second order optimality conditions for the optimal control problem. 

\begin{defn}\cite{cmj} $($\it{Local solution of the optimal control problem}$)$: The pair $(\bar \Psi, \bar u) \in  
\bV \times U_{ad}$ is a local solution of \eqref{wform} if and only if $(\bar \Psi, \bar u)$ satisfies \eqref{wv} and there exist neighbourhoods  ${\mathcal O}(\bar\Psi)$ of $\bar \Psi$ in $\bV$ and ${\mathcal O}(\bar u)$  of ${\bar u}$ in $L^2(\omega)$ such that 
$J({\bar \Psi},{\bar u} ) \le J(\Psi, u)$ for all pairs $(\Psi, u) \in {\mathcal O}(\bar\Psi)  \times (U_{ad} \cap {\mathcal O}(\bar u))$ satisfying \eqref{wv}.

\end{defn}

The existence of a solution of \eqref{wform} can be obtained using standard arguments of considering a minimizing sequence, which is bounded in $\bV \times L^2(\omega)$, and passing to the limit \cite{Lions1, hou, fredi2010}.

\medskip

For the purpose of numerical approximations, we consider only local solutions $(\bar \Psi, \bar u)$ of \eqref{wform} such that the pair is a nonsingular solution of \eqref{of}. For a {\it local nonsingular solution}  chosen in this fashion, we can apply Theorem \ref{th2.5} and modify the control problem \eqref{wform} to 
\begin{align} \label{rcf}
\inf_{u \in U_{ad} \cap {\mathcal O}(\bar u) } j(u),
\end{align}
 where $j:  U_{ad} \cap {\mathcal O}(\bar u) \rightarrow {\mathbb R}$ is the reduced cost functional defined by $j(u):= J(G(u), u)$ and $G(u) = \Psi_{u} =(\psi_{1u}, \psi_{2u})\in \bV$ is the unique solution to \eqref{wv} as defined in Theorem \ref{th2.5}. Then, ${\bar u}$ is a local solution of \eqref{rcf}.
 
 \medskip
Since $G$ is of class $C^{\infty}$ in   $ {\mathcal O}(\bar u)$, $j$ is of class $C^{\infty}$ and for every $u \in  {\mathcal O}(\bar u)$ and $v \in L^2(\omega)$, it is easy to compute
\begin{subequations} \label{derivatives}
\begin{align}
&j'(u)v=
\int_{\omega} \left( {\mathcal C}^*{\theta_{1u}} + \alpha { u} \right)  { v} \: dx, \label{fd} \\
&j''(u)v^2= \int_{\Omega} \left( |{\bf z}_v|^2+  { [[}{\bf z}_v, {\bf z}_v]]  \right) \Theta_{u} \: \dx + \alpha \int_{\omega} |v|^2 \: \dx, \label{sd}
\end{align} 
\end{subequations}
where ${\bf z}_v=(z_{1v}, z_{2v})$ is the solution of \eqref{E2.7}, 
 $${ [[}{\bf z}_v, {\bf z}_v]] := \big( [z_{1v},z_{2v}] + [z_{2v},z_{1v}], -[z_{1v},z_{1v}] \big), $$ $ [\cdot, \cdot]$ being the von  K\'{a}rm\'{a}n bracket,  $\Theta_{u} =(\theta_{1u}, \theta_{2u}) \in \bV$ is the solution of the adjoint system and 
 $${ [[}{\bf z}_v, {\bf z}_v]] \Theta_{u} :=\left( [z_{1v},z_{2v}]  + [z_{2v},z_{1v}]\right) \theta_{1u} -[z_{1v},z_{1v}] \theta_{2u}.$$ The adjoint system is given by
\begin{subequations} \label{adj}
\begin{align}
  &\Delta^2  \theta_1 -[\psi_{2u} ,\theta_{1u}] +[\psi_{1u} ,\theta_{2u}]= \psi_{1u}-\psi_{1d} \,\, \quad \mbox{ in } \Omega, \label{adj1_cts} \\
 &\Delta^2  \theta_2 -  [\psi_{1u} ,\theta_{1u}]  = \psi_{2u}-\psi_{2d}  \,\, \quad \mbox{ in } \Omega,  \label{adj2_cts} \\
 &\theta_1=0,\,\frac{\partial \theta_1}{\partial \nu}=0 \text{ and } \theta_2=0,\,\frac{\partial  \theta_2}{\partial \nu} =0
\text{  on }\partial\Omega. \label{adj3_cts}
\end{align}
\end{subequations}
As for the case of the state equations, the adjoint equations in \eqref{adj} can also be written
 equivalently in an operator form as 
 \begin{equation}\label{adjof}
 \Theta_{u} \in \bV \qquad {\mathcal A}^* \Theta_{u}+ {\mathcal B}'(\Psi_{u})^* \Theta_{u}= \Psi_{u}- \Psi_d \quad 
 \rm{in} \quad \bV', 
  \end{equation}
 with the operator $({\mathcal A}^* + {\mathcal B}'(\Psi_{u})^*)$ being an isomorphism from $\bV$ into ${\bV}'$ (see Theorem \ref{th2.5}).
 The first order optimality condition $j'({\bar u})(u- {\bar u}) \ge 0$ for all $ u \in U_{ad}$ translates to 
\begin{equation*}
\int_{\omega} \left({\mathbf C}^* {\bar \Theta} + \alpha {\bf \bar u} \right) \cdot ({\bf u} - {\bf \bar u}) \: \dx \ge 0 \qquad \forall   {\bf u} =(u,0), \;  {u \in U_{ad}},
\end{equation*}
where ${ \bf \bar u}= ({\bar u},0)$ and
$\bar \Theta=(\bar {\theta}_1,\bar {\theta}_2)$ being the adjoint state corresponding to  a local nonsingular solution $({\bar \Psi}, {\bar u}) \in \bV \times U_{ad}$ of \eqref{wform}, or equivalently in a scalar form as 
\begin{equation*}
\int_{\omega} \left({\mathcal C}^* {\bar {\theta}_1} + \alpha {\bar u} \right) ({ u} - { \bar u}) \: \dx \ge 0 \qquad \forall    {u \in U_{ad}}.
\end{equation*}

\medskip
\noindent The {\it optimality system} for the optimal control problem \eqref{wform} can be stated as follows:
\begin{subequations} \label{opt_con}
\begin{align}
&  A({\bar \Psi},\Phi)+B({\bar \Psi}, {\bar \Psi},\Phi)=(F + {\bf C}{\bar{\bf u}}, \Phi) \fl \Phi \in \bV \; \;  {(State \; equations)}\label{state_eq}
\\  
& A(\Phi, {\bar\Theta})+B({\bar \Psi}, \Phi, {\bar \Theta}) + B(\Phi, {\bar \Psi}, {\bar \Theta})=(\bar \Psi- \Psi_d, \Phi) \fl \Phi \in \bV \; \; {(Adjoint \; equations)}\label{adj_eq}
 \\
&  \left({\mathbf C}^* {\bar \Theta} + \alpha {\bf \bar u} ,  {\bf u} - {\bf \bar u} \right)_{{\bf L}^2(\omega)} \ge 0  \fl  {\bf u} =(u,0), \; {u \in U_{ad}}, \; \;  {(First \: order \: optimality \: condition)}. \label{opt3}
\end{align}

\end{subequations}
\noindent The optimal control ${\bf \bar u}$ in \eqref{opt3} has the representation for a.e. $x \in \Omega$ :
\begin{equation}\label{rep}
{ \bf \bar u}(x)= \pi_{[u_a, u_b]} \left( - \frac{1}{\alpha} ({\mathbf C}^* {\bar \Theta} ) \right),
\end{equation}
where the projection operator $\pi_{[a,b]}$ is defined by $\pi_{[a,b]}(g):= \min\{b, \max \{a,g\} \}.$
 
  \begin{rem}
   The optimality conditions in \eqref{opt_con}  can also be derived with the help of a Lagrangian for the constrained optimization problem  \eqref{wform} defined by 
   $$ L(\Psi, u, \Theta)= J(\Psi,u) -\left( A(\Psi, \Theta) + B(\Psi,\Psi, \Theta)-({\bf F} +{\mathbf C} {\bf u}, \Theta)\right) \quad \forall (\Psi,u, \Theta) \times \bV \times U_{ad} \times \bV. $$
   \end{rem} 
For the error analysis for this nonlinear control problem, second order sufficient optimality conditions are required. We now proceed to discuss the second order optimality conditions. 

\medskip
\noindent Define the {\it tangent cone} at ${\bar u} $ to $U_{ad}$ as 
\begin{equation*}
{\mathscr C}_{U_{ad}}({\bar u}) := \left\{ u \in L^2(\omega): u \; {\rm satisfies } \; \eqref{cases} \right\},
\end{equation*}
with 
\begin{align} \label{cases}
 \begin{cases}
u(x) \in {\mathbb R} &\text{if ${\bar u} (x)  \in (u_a,u_b)$},\\
u(x) \ge 0 &\text{if  ${\bar u} (x) =u_a$}, \\
u(x) \le 0 &\text{if  ${\bar u} (x) =u_b$}.
\end{cases} 
\end{align}

\noindent The function ${\mathcal C}^* {\bar \theta_1} + \alpha {\bar u}$ or $~{\mathbf C}^* {\bar \Theta} + \alpha {\bf \bar  u}~$ in the vector form, is used frequently in the analysis. Introduce the notation 
$$ {\bar d}(x)={\mathcal C}^* {\bar \theta_1} + \alpha {\bar u}, \; \; x \in \omega. $$
Associated with ${\bar d}$, we introduce another cone ${\mathscr C}_{\bar u} \subset {\mathscr C}
_{U_{ad}}({\bar u})$ defined by 
 
\begin{equation*}
{\mathscr C}_{\bar u} := \left\{ u \in L^2(\omega): u \; {\rm satisfies \;}  \eqref{cases1}\right\},
\end{equation*}
with
\begin{align} \label{cases1}
  \begin{cases}
u(x) =0 \mbox{ if } {\bar d}(x) \ne 0,\\
u(x) \ge 0  \mbox{ if } {\bar d}(x) = 0 \mbox{ and } {\bar u}(x)=u_a,\\
u(x) \le 0  \mbox{ if } {\bar d}(x) = 0 \mbox{ and } {\bar u}(x)=u_b.
\end{cases} 
\end{align}
By the definition of ${\bar d}$, we have
$$j'(u)v= \int_{\omega} {\bar d}(x)  v(x) \: dx \quad \forall v \in L^2(\omega). $$

\noindent Moreover, if we choose $v \in {\mathscr C}_{\bar u}$, the optimality condition \eqref{opt3} yields
${\bar d}(x) v(x) =0$ for almost all $x \in \omega$.

\medskip
\noindent The following theorem is on second order necessary optimality conditions. The proof  is on similar lines of the proof of Theorem 3.6  in \cite{cmj} and hence skipped.
\begin{thm}\label{ns}
Let $({\bar \Psi}, {\bar u})$ be a nonsingular local solution  of \eqref{wform}. 
\begin{equation}\label{second_opt_cond}
j''({\bar u})v^2 \ge 0 \qquad  \forall v \in {\mathscr C}_{\bar u}.
\end{equation}
\end{thm}
\begin{cor} Let $({\bar \Psi}, {\bar u})$ be a nonsingular local solution  of \eqref{wform} and ${\bar \Theta} = \Theta({\bar u})$ be the associated adjoint state. Then, 
Theorem \ref{ns} and \eqref{sd} give
$$ \int_{\Omega} \left( |{\bf \bar z}_v|^2+  { [[}{\bf \bar z}_v, {\bf \bar z}_v]]   {\bar \Theta}\right)  \dx + \alpha \int_{\omega} |v|^2  \dx > 0$$
for all $v \in {\mathscr C}_{\bar u}$, where ${\bf \bar z}_v= {\bz}_{v}({\bar u})$ is the solution to \eqref{E2.7} for $u= {\bar u}$ and $v \in {\mathscr C}_{\bar u}$.
\end{cor}

\begin{thm}[\it Second Order Sufficient Condition]\label{ssoc} Let $({\bar \Psi}, {\bar u})$ be a nonsingular local solution of \eqref{wform} and let ${\bar \Theta}=\Theta({\bar u})$ be the associated adjoint state. Assume that {}{$$ \int_{\Omega} \left( |{\bf \bar z}_v|^2+  { [[}{\bf \bar z}_v, {\bf \bar z}_v]]  {\bar \Theta} \right)  \dx + \alpha \int_{\omega} |v|^2 \dx  > 0$$}{} for all $v \in {\mathscr C}_{\bar u}$. Then, there exist $\epsilon >0$ and $\mu >0$ such that, for all $u \in U_{ad}$ satisfying, together with $\Psi_{u}$,  
$$ 
\|u - {\bar u}\|^2_{L^2(\omega)}  +\trinl \Psi_{u} - {\bar \Psi} \trinr^2\le \epsilon^2,
$$
we have 
$$
J({\bar \Psi}, {\bar u}) + 
\frac{\mu}{2} \left( \|u - {\bar u}\|^2_{L^2(\omega)} +
 \trinl \Psi_{u} - {\bar \Psi} \trinr^2 \right) 
\le J( \Psi_{u}, { u}) . 
$$ 
\end{thm}
\begin{proof} We argue by contradiction. If possible,  let $\{(\Psi_{u_k},u_k) \}$ be a sequence
satisfying  \eqref{wv} with $u_k \in U_{ad}$, such that 
\begin{align}
  \|u_k - {\bar u}\|^2_{L^2(\omega)}  +\trinl\Psi_{u_k} - {\bar \Psi} \trinr^2 \le \frac{1}{k^2} \label{assum}
  \end{align}
  and 
\begin{equation}\label{star}
J({\bar \Psi}, {\bar u}) + 
\frac{1}{k} \left( \|u_k - {\bar u}\|^2_{L^2(\omega)} +
 \trinl \Psi_{u_k} - {\bar \Psi} \trinr^2  \right) 
>  J( \Psi_{u_k}, { u_k}).
\end{equation}
Set 
$$\rho_k =\sqrt{{\|u_k - {\bar u}\|^2_{L^2(\omega)}  +\trinl \Psi_{u_k} - {\bar \Psi} \trinr^2 }} , \quad  v_k = \frac{u_k-{\bar u}}{\rho_k},\; {\bf z}_k = \frac{\Psi_{u_k}-{\bar \Psi}}{\rho_k}.$$
Note that $(\Psi_{u_k})_k$ is bounded in $ {\bH}^{2+\gamma}(\Omega)$, see \eqref{app2}. Clearly, 
$\|v_k\|^2_{L^2(\omega)} + \trinl {\bf z}_k \trinr^2=1$ and the pair 
$({\bf z}_k, v_k)$ satisfies the equation
\begin{equation}\label{lim}
 {\mathcal A} {\bf z}_k - \frac{1}{2} {\mathcal B}'({\bar \Psi}) {\bf z}_k + 
\frac{1}{2} {\mathcal B}'(\Psi_{u_k}) {\bf z}_k ={\mathcal C} v_k \; \mbox{ in  }  \bV'.
\end{equation}
Following the proof of Lemma \ref{linapriori}, we can verify that 
$\trinl {\bf z}_k \trinr_{2+\gamma} \le C \|v_k\|_{L^2(\omega)}, $ with a constant $C$ independent of $k$.  By passing to the limit (up to a subsequence) in \eqref{lim}, we can prove that
\begin{align*}
& {\bf z}_k \rightharpoonup {\bf z} \mbox{ in } {\bH}^{2+\gamma}(\Omega), \;
 {\bf z}_k \rightarrow {\bf z} \mbox{ in } \bV ,
v_k \rightharpoonup v \mbox{ in } L^2(\omega),
\end{align*}
and $ {\bf z}=  {\bf \bar z}_v$, that is, $ {\bf z}$ is the solution of \eqref{E2.7} associated with $v$ for $u=\bar{u}$.

\bigskip
\noindent Now we verify that $v \in {\mathscr C}_{\bar u}$. With \eqref{star}, we have 
\begin{eqnarray*}
\frac{\rho_k}{k} & >& \frac{J({\bar \Psi + \rho_k  {\bf z}_k}, {\bar u} + \rho_k v_k) - J ({\bar \Psi}, {\bar u})}{\rho_k} \\
& = & \frac{1}{2} \int_{\Omega}\left( 2 ({\bar \Psi} - \Psi_d) + \rho_k {\bf z}_k\right)
{\bf z}_k \: dx + \frac{\alpha}{2} \int_{\omega} \left( 2 {\bar u} + \rho_k v_k\right)  v_k \: dx.
\end{eqnarray*}
By passing to the limit as $k \rightarrow \infty$ and using \eqref{assum}, we obtain
$$ \int_{\Omega}  ({\bar \Psi} - \Psi_d) {\bf z}_v \: dx + {\alpha} \int_{\omega}   {\bar u} v dx \le 0,$$
which yields $\displaystyle \int_{\omega} {\bar d}(x) v(x) \: dx \le 0.$ The last condition implies that $v \in {\mathscr C}_{\bar u}$.

\medskip
\noindent Making a second order Taylor expansion of $J$ at $({\bar \Psi}, {\bar u})$, we have 
\begin{align*}
J(\Psi_{u_k}, u_k) = & J({\bar \Psi}, {\bar u}) + \partial_{\Psi}  J({\bar \Psi}, {\bar u}) \: \rho_k {\bf z}_k +
\partial_{u}  J({\bar \Psi}, {\bar u}) \: \rho_k {v}_k \\
&  +
\frac{1}{2} \int_{\Omega} |\Psi_{u_k} - {\bar \Psi}|^2\: dx + \frac{\alpha}{2} \int_{\omega} |u_k-{\bar u}|^2\: dx.
\end{align*}
Thus with \eqref{star}, we can write 
\begin{align} \label{2star}
\partial_{\Psi}  J({\bar \Psi}, {\bar u}) {\bf z}_k +
\partial_{u}  J({\bar \Psi}, {\bar u})  {v}_k + \frac{1}{2} \int_{\Omega} |{\bf z}_k|^2 \: dx + 
\frac{\alpha}{2} \int_{\omega} |v_k|^2 \: dx < \frac{1}{k}.
\end{align}
Also,
\begin{align*}
\partial_{\Psi}  J({\bar \Psi}, {\bar u}) {\bf z}_k +
\partial_{u}  J({\bar \Psi}, {\bar u})  {v}_k  = \int_{\Omega} ({\bar \Psi}- \Psi_d) {\bf z}_k \: dx + 
\alpha \int_{\omega} {\bar u} v_k \: dx,
\end{align*}
and using the adjoint state ${\bar \Theta}$, we obtain
\begin{align*}
\int_{\Omega}({\bar \Psi}- \Psi_d) {\bf z}_k \: dx = \int_{\Omega} {\bar \Theta}\left({\mathcal C}v_k -
{\mathcal B}'({\bf z}_k)(\Psi_k-{\bar \Psi}) \right) \: dx.
\end{align*}
Thus,
\begin{align*}
\frac{1}{\rho_k}\left( \partial_{\Psi}  J({\bar \Psi}, {\bar u}) {\bf z}_k +
\partial_{u}  J({\bar \Psi}, {\bar u})  {v}_k\right) = \frac{1}{\rho_k}
\int_{\Omega} {\bar \Theta}\left( {\bar d}(x) v_k \: dx - \int_{\Omega}
{\mathcal B}'({\bf z}_k)({\bf z}_k) \right) {\bar \Theta}\:  dx.
\end{align*}
Since ${\bar d}(x) v(x) \ge 0 $, with \eqref{2star}, we have 
$$ - \int_{\Omega}
2{\mathcal B}'({\bf z}_k)({\bf z}_k)  {\bar \Theta}\:  dx + \int_{\Omega}|{\bf z}_k|^2 \: dx + 
\alpha \int_{\omega} |v_k|^2 \: dx < \frac{2}{k}.$$
By passing to the inferior limit, we have 
$$\int_{\Omega}  \left( |{\bf z}|^2 + [[{\bf z}, {\bf z}]] \: {\bar \Theta} \right) \: dx + \alpha \int_{\omega} |v|^2 \: dx \le 0.$$

Since $v \in {\mathscr C}_{\bar u}$ and due to our assumption about the {\bf sufficient} second order optimality condition, we have $({\bf z}, v)=(0,0)$.
Hence,
\begin{align*}
& \lim_{k \rightarrow \infty} \left( -2 \int_{\Omega} {\mathcal B}({\bf z}_k) {\bar \Theta}   + \int_{\Omega} |{\bf z}_k|^2 \: dx\right) 
=  -2 \int_{\Omega} {\mathcal B}({\bf z}) {\bar \Theta}  \: dx + \int_{\Omega} |{\bf z}|^2 \: dx,
\end{align*}
  and thus $\displaystyle \lim_{k \rightarrow \infty} \int_{\omega} |v_k|^2 \: dx =0$. Thus we have a contradiction with $\|v_k\|^2_{L^2(\omega)} + \|{\bf z}_k\|^2_{L^2(\Omega)}=1,$ and the proof 
  is complete.
\end{proof}

\noindent Note that the second order optimality condition 
$$\int_{\Omega}  \left( |{\bf \bar z}_{ v}|^2 + [[{\bf \bar z}_v, {\bf \bar z}_v]] \: {\bar \Theta} \right) \: dx + \alpha \int_{\omega} |v|^2 \: dx >0, \; \;$$ for all $v \in {\mathcal C}_{\bar u}$ is equivalent to $j''({\bar u})v^2 >0 \quad \forall v \in {\mathscr C}_{\bar u}$.

\bigskip
\noindent As in \cite{cmj}, we reinforce the above condition by  assuming that 
\begin{align}\label{ssoc2}
j''({\bar u})v^2 > \delta \left( \|v\|^2_{L^2(\omega) } + \|{}{{\bf \bar z}_v}{} \|^2_{L^2(\Omega)} \right),
\quad \forall v \in {\mathscr C}_{\bar u}^{\tau}, 
\end{align}
where
\begin{equation*}
{\mathscr C}_{\bar u}^{\tau} := \left\{ v \in L^2(\omega): \eqref{cases3} \mbox{ is satisfied} \right\},
\end{equation*}
with
\begin{align} \label{cases3}
 \begin{cases}
v(x) =0 \mbox{ if } | {d}(x) | > \tau,\\
v(x) \ge 0  \mbox{ if }    | {d}(x) | \le \tau \mbox{ and } {\bar u}(x)=u_a,\\
v(x) \le 0  \mbox{ if }  | {d}(x) | \le \tau \mbox{ and } {\bar u}(x)=u_b.
\end{cases} 
\end{align}
and ${\bf \bar z}_v$ is the solution of \eqref{E2.7} {}{with $u={\bar u}$}{}.

\begin{thm}(Theorem 3.10, \cite{cmj}) The condition \eqref{second_opt_cond} is equivalent to \eqref{ssoc2}.
\end{thm}

\section{Discretization of State \& Adjoint Variables}

In this section, first of all, we describe the discretization of the state variable using {conforming} finite elements. This is followed by definition of an auxiliary discrete problem  corresponding to the state equation for a given control $u \in U_{ad}$. We establish  the existence of  a unique solution and error estimates  for this problem under suitable assumptions. Similar results for an auxiliary problem corresponding to the adjoint variable is proved next.

\subsection{Conforming finite elements}
Let $\mathcal T_h$ be a regular, conforming and quasi-uniform triangulation of $\overline{\Omega}$ into closed triangles, rectangles or quadrilaterals. Set $h_T=diam(T)$, $T \in \mathcal T_h $ and define the discretization parameter $h:=\max_{T\in\mathcal{T}_h}h_T$. We now provide examples of  two conforming finite elements defined on a triangle and a rectangle,  namely the Argyris elements and \BFS (see Figure ~\ref{fig:elements}).

\begin{figure}[h!]
  \begin{center}
   \subfloat[]{\includegraphics[width = 2.5in,height=2in]{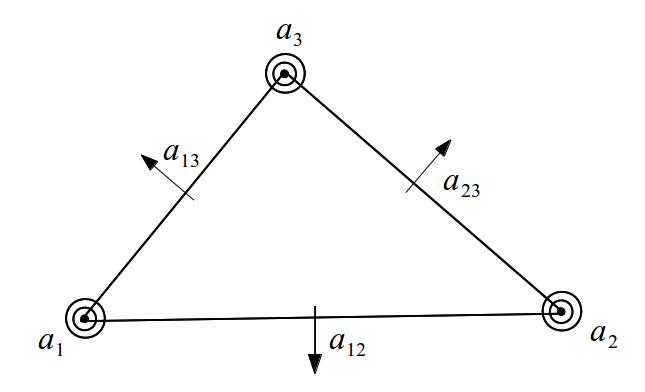}}
    \subfloat[]{\includegraphics[width = 2.4in,height=2in]{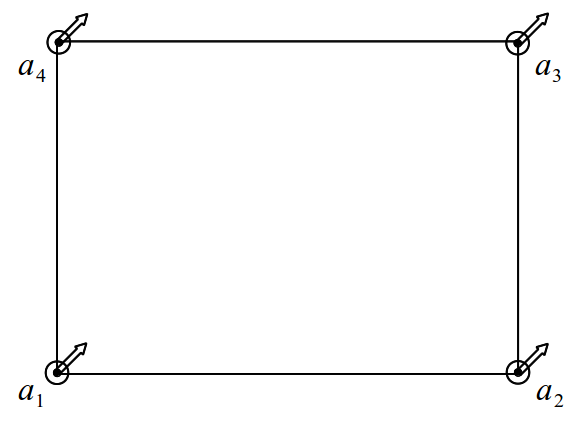}}
    \caption{ (A) Argyris element and (B) \BFS element}\label{fig:elements}
  \end{center}
\end{figure}
\begin{defn}[Argyris element \cite{Brenner,Ciarlet}]
The Argyris element is a triplet $(T,P_5(T),\Sigma_T)$ where $T$ is a triangle,  $P_5(T)$ denotes polynomials of degree $\leq 5$ in both the variables and the 21 degrees of freedom in $ \Sigma_T$ are determined by the values of the unknown functions, its  first order and second order derivatives at the three vertices and the normal derivatives at the midpoints of the three edges of $T$ (see Figure \ref{fig:elements}(A)).
\end{defn}
\begin{defn}[\BFS element \cite{Ciarlet}] Let $T\in \mathcal{T}_h$ be a rectangle with vertices $a_i=(x_i,y_i),\: i=1,2,3,4$. 
 The \BFS element is a triplet $(T, Q_3(T),\Sigma_T)$, where
 $Q_3(T)$ denotes polynomials of degree $\leq 3$ in both the variables and the degrees of freedom $\Sigma_T$ is defined by  $\Sigma_T=\{p(a_i),\frac{\partial p}{\partial x}(a_i),\frac{\partial p}{\partial y}(a_i),\frac{\partial^2 p}{\partial x\partial y}(a_i),\; 1\leq i\leq 4\}$ (see Figure \ref{fig:elements}(B)).
\end{defn}

The conforming  $C^1$ finite element spaces associated with Argyris and \BFS elements are contained in
$ C^1(\overline\Omega)\cap H^2(\Omega)$. Define
\begin{align*}
  V_h&=\left\{v\in C^1(\overline\Omega):v|_{T}\in P_T,\fl T\in\mathcal{T}_h \text{ with }v|_{\partial\Omega}=0, \frac{\partial v}{\partial \nu}\big{|}_{\partial\Omega}=0\right\}\subset  H^2_0(\Omega),
\end{align*}
where
\begin{align*} 
    P_T &= \left\{ 
      \begin{array}{l l}
        P_5(T) & \text{for Argyris element}, \\
        Q_3(T)& \text{for \BFS element}.
      \end{array} \right.
\end{align*}
The discrete state and adjoint variables are sought in the finite dimensional space defined by $\bV_h:={V_h} \times V_h$.

\medskip

\begin{lem}[\it Interpolant \cite{Ciarlet}]\label{interpolant}
Let $\Pi_h:V\map V_h$ be the Argyris or \BFS nodal interpolation operator. 
Then for $\varphi\in H^{2+\gamma}(\Omega),$ with $\gamma\in(\half,1]$ denoting the index of elliptic regularity, it holds:
\begin{align} \label{inte}
&\|\varphi-\Pi_h\varphi\|_{m} \leq C h^{\min\{k+1,\;2+\gamma\}-m}\|\varphi\|_{2+\gamma} \qquad\text{ for } m=0,1,2;
\end{align}
where $k=5$ (resp. 3) for the Argyris element  (resp. \BFS element).
\end{lem}

\subsection{Auxiliary problems for the state equations}
Define  an auxiliary continuous  problem associated with the state equation as follows:

\bigskip

\noindent Seek $\Psi_{u}\in \bV$ such that 
  \begin{align} 
   &  A(\Psi_{u},\Phi)+B(\Psi_{u},\Psi_{u},\Phi)=(F + {\mathbf C}{\bf u}, \Phi) \; \fl \Phi \in \bV, \label{wv1}
  \end{align}
where  ${\bf u}=(u,0)$, $u \in L^2(\omega)$ is given. 

\medskip

\noindent A discrete conforming finite element approximation for this problem can be defined as:

Seek $\Psi_{u,h}\in \bV_h$ such that 
  \begin{align} 
   &  A(\Psi_{u,h},\Phi_h)+B(\Psi_{u,h},\Psi_{u,h},\Phi_h)=(F + {\mathbf C}{\bf u}, \Phi_h) \; \fl \Phi_h \in \bV_h. \label{wform2}
  \end{align}

\noindent For a given $u \in L^2(\omega)$, \eqref{wform2} is not well-posed in general. 
The main results  of this subsection are stated now.
\begin{thm} \label{aux_dis}
Let $({\bar \Psi}, {\bar u}) \in \bV \times L^2(\omega) $ be a nonsingular solution of \eqref{of}. 
Then, there exist $\rho_1, \rho_2>0$ and $h_1 >0$ such that, for all $0 < h< h_1$ and ${u} \in B_{\rho_2}({\bar u})$, \eqref{wform2} admits a unique solution in 
$B_{\rho_1}(\bar \Psi)$.
\end{thm}
\begin{rem}
For $\rho>0$, ${u} \in B_{\rho}({\bar u})$ means that $\|u - {\bar u}\|_{L^2(\omega)} \le \rho.$ Similarly, ${\Psi} \in B_{\rho}({\bar \Psi}) \Longrightarrow \trinl \Psi - {\bar \Psi} \trinr_{2} \le \rho.$
\end{rem}

\begin{thm} \label{ee1}
 Let $({\bar \Psi}, {\bar u}) \in \bV \times L^2(\omega)$ be a nonsingular solution of \eqref{of}.  Let $h_1$ and $\rho_2$ be defined as in  
Theorem \ref{aux_dis}. 
Then, for ${u} \in B_{\rho_2}({\bar u})$ and $0 < h< h_1$, the solutions $\Psi_{u}$ and $\Psi_{u,h}$ of \eqref{wv1} and \eqref{wform2} satisfy the error estimates:
\begin{equation}\label{ee1_est}
(a) \; \trinl\Psi_{u}-\Psi_{u,h}\trinr_{2} \le C h^\gamma \qquad (b) \;  \; \trinl\Psi_{u}-\Psi_{u,h}\trinr_{1} \le C h^{2\gamma},
\end{equation} 
where $\gamma \in (1/2,1]$ denotes the index of elliptic regularity.
\end{thm}
\noindent We proceed to establish several results which  will be essential to prove  Theorem \ref{aux_dis}. The proof of Theorem \ref{ee1} follows from the error estimates for the approximation of \vket using conforming finite element methods; see \cite{Brezzi,ng1}.
 
 \bigskip
 \noindent{\bf An auxiliary linear problem and discretization}

 \medskip
 For a given ${\bf g}=(g_1,g_2) \in \bV'$, let $T \in {\mathcal L}(\bV', \bV)$ be defined by 
 $T {\bf  g}:= \boldsymbol{\xi} =(\xi_1, \xi_2) \in \bV$ where $\boldsymbol{\xi}$ solves the system of biharmonic equations given by:
 \begin{subequations} \label{bih}
\begin{align}
&  \Delta^2  \xi_1 = g_1 \,\, \quad \mbox{ in } \Omega, \label{state_bih1} \\
&  \Delta^2  \xi_2 = g_2 \quad \mbox{ in } \Omega,  \label{state_bih2} \\
& \xi_1=0,\,\frac{\partial \xi_1}{\partial \nu}=0\text{ and } \xi_2=0,\,\frac{\partial  \xi_2}{\partial \nu} =0
\text{  on }\partial\Omega. \label{bc3}
\end{align}
\end{subequations}
 \noindent Equivalently, $\boldsymbol{\xi} \in \bV$ solves $A(\boldsymbol{\xi}, \Phi) = \big\langle {\bf g}, \Phi \big\rangle_{\bV', \bV} \quad \forall \Phi \in \bV$.
 
 \medskip
\noindent Also, let $T_h \in {\mathcal L}(\bV', \bV_h)$ be defined by 
 $T_h {\bf  g}:= \boldsymbol{\xi}_h $ if $\boldsymbol{\xi}_h \in \bV_h$ solves the discrete problem
  \begin{equation} \label{dis_bih}
  A(\boldsymbol{\xi}_h, \Phi_h) = \big\langle {\bf g}, \Phi_h \big\rangle_{\bV', \bV} \quad \forall \Phi_h \in \bV_h.
  \end{equation}
 \begin{lem}{\it (A bound for $T_h$)} \label{bd_Th}
 There exists a constant $C>0$, independent of $h$, such that 
 $$\|T_h\|_{{\mathcal L}(\bV', \bV)} \le C. $$
 \end{lem}
\begin{proof}
The definition of $T_h{\bf g}$ along with coercivity property of the bilinear form $A(\cdot,\cdot)$ lead to the required result.
\end{proof}
\begin{lem}\label{BihErrEst}\cite{Brenner}({\it Error estimates})
Let $\boldsymbol{\xi}$ and $\boldsymbol{\xi}_h$ solve \eqref{bih} and \eqref{dis_bih} respectively. Then it holds:
  \begin{subequations} \label{sub}
  \begin{align}
 &\trinl \boldsymbol{\xi}-\boldsymbol{\xi}_h \trinr_{2} \le C h^{\gamma} \trinl{\bf g}\trinr_{-1} \quad \forall {\bf g} \in  {\bH}^{-1}(\Omega), \\
  &\trinl \boldsymbol{\xi}-\boldsymbol{\xi}_h \trinr_{1} \le C h^{2\gamma} \trinl{\bf g}\trinr_{-1} \quad \forall {\bf g} \in  {\bH}^{-1}(\Omega),
  \end{align}
  \end{subequations}
  $\gamma \in (1/2,1]$ being the index of elliptic regularity. That is, $\trinl (T-T_h){\bf g} \trinr_{2} \le C h^{\gamma} \trinl{\bf g}\trinr_{-1}$ and
  $\trinl (T-T_h){\bf g} \trinr_{1} \le C h^{2\gamma} \trinl{\bf g}\trinr_{-1}$.
  \end{lem}
  \begin{rem} When ${\bf g}=(g,0)$, we denote  $T{\bf g}$ (resp. $T_h{\bf g}$) as $Tg$ (resp. $T_hg$), purely for notational convenience.
  \end{rem}
  
  \bigskip
 \noindent{\bf A nonlinear mapping and its properties}

  \medskip
  Define a nonlinear mapping ${\mathscr N}: \bV \times L^2(\omega) \rightarrow \bV$ by 
  $$ {\mathscr N} (\Psi, u):= \Psi + T[{\mathcal B}(\Psi)-(F + {\mathbf C}{\bf u})], \qquad {\bf u} =(u,0).$$
  
  \smallskip
 \noindent  Now ${\mathscr N} (\Psi, u)=0$ if and only if  $(\Psi, u)$ solves \eqref{of}; that is,
  ${\mathcal A} \Psi + {\mathcal B}(\Psi) = F + {\mathbf C}{\bf u} \quad {\rm in} \; \bV'.$
  
  \medskip
 \noindent Similarly, define   a nonlinear mapping ${\mathscr N}_h: \bV \times L^2(\omega) \rightarrow \bV$ by 
  $$ {\mathscr N}_h (\Psi, u):= \Psi + T_h[{\mathcal B}(\Psi)-(F + {\mathbf  C}{\bf u})], \qquad {\bf u} =(u,0).$$
 Note that,  ${\mathscr N}_h (\Psi, u)=0$ if and only if  $\Psi \in \bV_h$ and $\Psi= \Psi_{u,h}$ solves \eqref{wform2}.

\medskip
\noindent The derivative mapping $\partial_\Psi  {\mathscr N} (\Psi, u)  \; (\mbox{resp.} \; \partial_\Psi  {\mathscr N}_h (\Psi, u)) \in {\mathcal L}(\bV)$  is defined by 
$$\partial_\Psi  {\mathscr N} (\Psi, u)(\Phi) = \Phi + T[{\mathcal B}'(\Psi) \Phi] \quad \forall \Phi \in \bV.$$
$$\big(\mbox{resp.} \; \partial_\Psi  {\mathscr N}_h (\Psi, u)(\Phi) = \Phi + T_h[{\mathcal B}'(\Psi) \Phi] \quad \forall \Phi \in \bV.\big)$$

\medskip

With definitions of nonsingular solution (see Definition \ref{isolated}), the linear mapping $T$ and the derivative mapping $\partial_\Psi  {\mathscr N} (\Psi, u)$, we obtain the following result, the proof of which is skipped. 
\begin{lem}\cite{cmj} If $({\bar \Psi}, {\bar u}) \in \bV \times L^2(\omega)$ is a nonsingular solution of \eqref{wv}, then $\partial_\Psi  {\mathscr N} (\bar \Psi, \bar u)$ is an automorphism in $\bV$. The converse also holds true.
\end{lem}
We want to establish that if $({\bar \Psi}, \bar u)$ is a nonsingular solution, then the derivative mapping  $\partial_\Psi  {\mathscr N}_h (\cdot, \cdot)$ is an automorphism in $\bV$, with respect to small perturbations of its arguments. That is, if $\trinl { \Psi}- {\bar \Psi} \trinr_2$ and $\|u - \bar u\|_{L^2(\omega)} $ are small enough, then   $\partial_\Psi  {\mathscr N}_h (\Psi, u)$  is  an automorphism in $\bV$. 
The next two lemmas will be useful in proving this result.
\begin{lem} \label{aux}Let $\bar \Psi \in \bV$ be a nonsingular solution of \eqref{of}. Then,  $\forall \epsilon >0 $, $\exists h_{\epsilon}>0$  such that 
\begin{equation}\label{tes}
\| T[B'(\bar \Psi)] - T_h [B'(\Psi)]\|_{{\mathcal L}(\bV)} < \epsilon \quad \forall \Psi \in B_{\rho_\epsilon}(\bar \Psi),
\end{equation}
whenever $0 < h < h_\epsilon$.
\end{lem}
\begin{proof}
For a fixed ${\bf z} \in \bV$, let $T[B'(\bar \Psi) {\bf z}]=: {\boldsymbol \theta}(\bar \Psi) \in \bV$
and $T_h[B'( \Psi) {\bf z}]=: {\boldsymbol \theta}_h(\Psi) \in \bV_h$.
Then ${\boldsymbol \theta}(\bar \Psi)$ and ${\boldsymbol \theta}_h( \Psi)$, respectively solve 
\begin{align}
 A({\boldsymbol \theta}(\bar \Psi), \Phi)&=\big \langle B'(\bar  \Psi) {\bf z}, \Phi \big \rangle_{\bV', \bV} =  B({\bar \Psi}, {\bf z}, \Phi) +  B({\bf z}, {\bar \Psi},  \Phi)  \quad \forall \Phi \in \bV, \label{a1}\\
A({\boldsymbol \theta}_h( \Psi), \Phi_h)& = \big \langle B'( \Psi) {\bf z}, \Phi_h \big \rangle_{\bV', \bV} = B({ \Psi}, {\bf z}, \Phi_h) +  B({\bf z}, { \Psi},  \Phi_h)  \quad \forall \Phi_h \in \bV_h.  \label{a2}
\end{align}
Let ${\boldsymbol \theta}_h({\bar \Psi}) \in \bV_h$ be the solution to the intermediate problem defined by 
\begin{align}
& A({\boldsymbol \theta}_h(\bar \Psi), \Phi_h)= \big \langle B'(\bar \Psi) {\bf z}, \Phi_h \big \rangle_{\bV', \bV}= B({\bar \Psi}, {\bf z}, \Phi_h) +  B({\bf z}, {\bar \Psi},  \Phi_h)  \quad \forall \Phi_h \in \bV_h. \label{a3} 
\end{align}
The triangle inequality yields  
\begin{equation}\label{trian}
\trinl {\boldsymbol \theta}(\bar \Psi) -{\boldsymbol \theta}_h( \Psi) \trinr_2 \le  
\trinl {\boldsymbol \theta}(\bar \Psi) -{\boldsymbol \theta}_h( \bar \Psi) \trinr_2 + 
\trinl {\boldsymbol \theta}_h(\bar \Psi) -{\boldsymbol \theta}_h( \Psi) \trinr_2.
\end{equation}
To estimate the first term in the right hand side of \eqref{trian}, consider \eqref{a1} and \eqref{a3}; use the facts that $\bV_h \subset \bV$, the error $({\boldsymbol \theta}(\bar \Psi) -{\boldsymbol \theta}_h( \bar \Psi) )$
is orthogonal to $\bV_h$ in the energy norm, the coercivity of $A(\cdot, \cdot)$, the interpolation estimate given in Lemma \ref{interpolant} and the fact that ${\bar \Psi} \in {\bH}^{2 +\gamma}(\Omega)$ to obtain 
\begin{equation}\label{trian1}
\trinl {\boldsymbol \theta}(\bar \Psi) -{\boldsymbol \theta}_h( \bar \Psi) \trinr_2 \le C h^{\gamma}   \trinl  {\boldsymbol \theta}(\bar \Psi) \trinr_{2 + \gamma} \le 
  C h^{\gamma}   \trinl  {\mathcal B}'(\bar \Psi) {\bf z} \trinr_{-1}.
\end{equation}
From definition of $ {\mathcal B}'(\bar \Psi) {\bf z} $, \eqref{boundB2} and the fact that $B(\cdot,\cdot,\cdot)$ is symmetric in first and second variables, it follows that
\begin{align}
\trinl  {\mathcal B}'(\bar \Psi) {\bf z} \trinr_{-1}  & = \sup
 \frac{ \left|  \big \langle {\mathcal B}'(\bar \Psi) {\bf z}, \Phi    \big\rangle \right| } 
{ \trinl \Phi \trinr_{1}} \nonumber \\
& \le  \sup
 \frac{  |B({\bar \Psi}, {\bf z}, \Phi) +  B({\bf z}, {\bar \Psi},  \Phi)|   }
{ \trinl \Phi \trinr_{1}}  \nonumber \\
& \lesssim \trinl \bar \Psi \trinr_{2+\gamma} \trinl {\bf z} \trinr_{2}. \label{bdash}
\end{align}
A substitution of \eqref{bdash} in \eqref{trian1} leads to 
\begin{equation}\label{trian2}
\trinl {\boldsymbol \theta}(\bar \Psi) -{\boldsymbol \theta}_h( \bar \Psi) \trinr_2 \le
 C h^{\gamma}  \trinl \bar \Psi \trinr_{2} \trinl {\bf z} \trinr_{2}.
 \end{equation}
 To estimate the second term on the right hand side of \eqref{trian}, subtract \eqref{a2} and \eqref{a3}, choose $\Phi_h=  {\boldsymbol \theta}_h(\bar \Psi) -{\boldsymbol \theta}_h( \Psi) $, use \eqref{coercivity} and \eqref{boundB1} to obtain

 \begin{equation} \label{trian3}
 \trinl {\boldsymbol \theta}_h(\bar \Psi) -{\boldsymbol \theta}_h( \Psi) \trinr_2 
 \le C \trinl    \bar \Psi - {\Psi} \trinr_2 \trinl {\bf z}\trinr_2.
 \end{equation} 
 A use of \eqref{trian2} and \eqref{trian3} in \eqref{trian} leads to the required result, when $h_\epsilon$ and $\rho_\epsilon$ are chosen sufficiently small.
\end{proof}
\noindent The next lemma is a standard result in Banach spaces and hence we refrain from providing a proof.
\begin{lem} \label{banach}
Let $X$ be a Banach space,
$A\in{\mathcal{L}}(X)$ be  invertible  and $B\in{\mathcal{L}}(X)$. If
$\|A-B\|_{{\mathcal{L}}(X)}<1/\|A^{-1}\|_{{\mathcal{L}}(X)}$, then
$B$ is invertible. If
$\|A-B\|_{{\mathcal{L}}(X)}<1/(2\|A^{-1}\|_{{\mathcal{L}}(X)})$ then
$\|B^{-1}\|_{{\mathcal{L}}(X)}\leq 2\|A^{-1}\|_{{\mathcal{L}}(X)}$. 
\end{lem}
\noindent The following theorem is a consequence of Lemmas \ref{aux}-\ref{banach}.
\begin{thm}[\it Invertibility of $\partial_{\Psi} {\mathscr N}_{h}(\Psi,u)$] \label{invert}
Let 
$({\bar \Psi}, {\bar u}) \in \bV \times L^2(\omega)$ be a nonsingular solution of \eqref{wv}. Then, there exist $h_0>0$ and
$\rho_0>0$ such that, for all $0<h<h_0$ and all ${ \Psi}\in
B_{\rho_0}(\bar \Psi)$, $ {\mathscr N}_h ( \Psi,  u)$ is an automorphism in $\bV$ and 
\[\|{\partial_{\Psi}{\mathscr{N}}_h}({{\Psi}},{{u}})^{-1}
\|_{{\mathcal{L}}(\bV)}\leq
2 \|{\partial_{\Psi}{\mathscr{N}}}({{\bar \Psi}},{{\bar u}})^{-1}
\|_{{\mathcal{L}}(\bV)}.\]
\end{thm}
\begin{proof}
Choose $h_0:= h_{\epsilon}, \; \rho_0:= \rho_{\epsilon} $ and $\epsilon= \frac{1}{2 \|{\partial_{\Psi}{\mathscr{N}}}({{\bar \Psi}},{{\bar u}})^{-1}
\|_{{\mathcal{L}}(\bV)}}$ in Lemma \ref{aux}.

\medskip
For every $0 < h< h_0$ and for all $\Psi \in B_{\rho_0}({\bar \Psi})$, the definitions of the derivatives of 
$\partial_\Psi {\mathscr N}$, $\partial_\Psi {\mathscr N}_h$ and \eqref{tes} yield
\begin{align*}
\|{\partial_{\Psi}{\mathscr{N}}}({{\bar \Psi}},{{\bar u}}) -{\partial_{\Psi}{\mathscr{N}_h}}({{ \Psi}},{{ u}})
\|_{{\mathcal{L}}(\bV)} & = \| T[B'(\bar \Psi)] - T_h [B'(\Psi)]\|_{{\mathcal L}(\bV)} < \epsilon. 
\end{align*}
Now, an application of Lemma \ref{banach} yields the required result.
\end{proof}
\noindent We now proceed to provide a proof Theorem \ref{aux_dis}, which is the main result of this subsection.

\medskip
\noindent{\it \textbf {Proof of Theorem \ref{aux_dis}}}:  

\medskip 
\noindent Let $({\bar \Psi}, {\bar u}) \in \bV \times L^2(\omega)$ be a nonsingular solution of \eqref{of}. 

\medskip
\noindent We need to establish that there exist $\rho_1, \rho_2>0$ and $h_1 >0$ such that, for all $0 < h< h_1$ and ${u} \in B_{\rho_2}({\bar u})$,  ${\mathscr N}_h(\Psi,u) =0$ admits a unique solution in 
$B_{\rho_1}(\bar \Psi)$.

\medskip
\noindent Let $\rho_0$ and $h_0$ be the positive constants as defined in Theorem \ref{invert}.  For $\rho \le \rho_0$, $h \le h_0$ and $u \in B_{\rho_2}({\bar u})$, define a mapping ${\mathscr G}(\cdot,u): B_{\rho}({\bar \Psi}) \rightarrow \bV$ by 
$${\mathscr G}(\Psi,u) = \Psi - [\partial_{\Psi} {\mathscr N}_h({\bar \Psi}, {\bar u})]^{-1} {\mathscr N}_h(\Psi, u).$$
Any fixed point of ${\mathscr G}(\cdot,u)$ is a solution of the discrete nonlinear problem ${\mathscr N}_h(\Psi, u)=0$.
In the next two steps, we establish that   (i)  ${\mathscr G}(\cdot, u)$ maps $B_{\rho}({\bar \Psi}) $ into itself; and  (ii) ${\mathscr G}(\cdot, u)$ is a strict contraction, if  $\rho$ is small enough.
\medskip

\noindent{\it Step 1:} The definition of ${\mathscr G}(\cdot, u)$, an addition of the zero term ${\mathscr N}({\bar \Psi}, {\bar u})$, an addition and subtraction of an intermediate  term 
and the Taylor's Theorem yield 
\begin{align}
\trinl {\mathscr G}(\Psi,u)- {\bar \Psi} \trinr_2 & = 
\trinl (\Psi- {\bar \Psi}) - [\partial_{\Psi} {\mathscr N}_h({\bar \Psi}, {\bar u})]^{-1} {\mathscr N}_h(\Psi, u) \trinr_2 \nonumber \\
& \le  \trinl [\partial_{\Psi}  {\mathscr N}_h({\bar \Psi}, {\bar u})]^{-1}  \left\{
[\partial_{\Psi} {\mathscr N}_h({\bar \Psi}, {\bar u})]  (\Psi- {\bar \Psi}) -
[{\mathscr N}_h(\Psi, u)-{\mathscr N}_h({\bar \Psi}, u)  ]  \right\} \trinr_2 \nonumber \\
 &\qquad + \trinl [\partial_{\Psi}  {\mathscr N}_h({\bar \Psi}, {\bar u})]^{-1}   [{\mathscr N}(\bar \Psi, \bar u) - {\mathscr N}_h(\bar \Psi, u) ]  
\trinr_2. \label{int1}
\end{align}

A use of Theorem \ref{invert}, Taylor formula for the second expression in the first term of the right hand side of \eqref{int1} along with the fact that the expression for the derivative $\partial_{\Psi}$ is independent of $u$ yields for $\Psi_{t}= \bar \Psi + t(\Psi - \bar \Psi), \; 0 < t < 1$, 
\begin{align}
\trinl {\mathscr G}(\Psi,u)- {\bar \Psi} \trinr_2 &  \le 
C \trinl \partial_{\Psi} {\mathscr N}_h({\bar \Psi}, {\bar u})]  (\Psi- {\bar \Psi}) -
\int_0^1 \partial_{\Psi}  {\mathscr N}_h({ \Psi_t},  u) (\Psi- {\bar \Psi})   \trinr_2 \nonumber \\
& \quad + C \trinl    [{\mathscr N}(\bar \Psi, \bar u) - {\mathscr N}_h(\bar \Psi, u) ]  
\trinr_2. \nonumber 
\end{align} 
With definitions of ${\mathscr N}(\cdot, \cdot)$ and ${\mathscr N}_h(\cdot,\cdot)$, and the triangle inequality in the above expression, we obtain
\begin{align} \label{main}
\trinl {\mathscr G}(\Psi,u)- {\bar \Psi} \trinr_2   \le& 
C \int_{0}^1 \trinl \partial_{\Psi} {\mathscr N}_h({\bar \Psi}, {\bar u})] - 
\partial_{\Psi}  {\mathscr N}_h({ \Psi_t}, u)  \trinr_2 \: dt \times \trinl \Psi- {\bar \Psi} \trinr_2 \nonumber \\
&+ \; C \trinl (T-T_h) ({\mathcal B}(\bar \Psi) - f)\trinr_2 + C \trinl (T-T_h) ({\mathcal C}{\bar u}) \trinr_2 \nonumber \\
& + \; C \trinl T_h ({\mathcal C}{\bar u} - {\mathcal C}u) \trinr_2=: T_1 + T_2+ T_3+ T_4.
\end{align}
We now estimate the terms $T_1$ to $T_4$.
With the definition of $\partial_{\Psi}{\mathscr N}_h(\cdot,\cdot)$, Lemma \ref{bd_Th}, the definition of ${\mathcal B}'(\cdot)$ and  \eqref{boundB1}, it yields 
\begin{align}
\| \partial_{\Psi} {\mathscr N}_h({\bar \Psi}, {\bar u}) - 
\partial_{\Psi}  {\mathscr N}_h({ \bar \Psi + t(\Psi - \bar \Psi)}, u)  \|_{\cL(V)} \:
&= \|  T_h({\mathcal B}'({\bar \Psi} + t (\Psi- \bar \Psi))) - T_h({\mathcal B}'(\bar \Psi)) \|_{\cL(V)} \nonumber \\
& \le C \trinl \Psi- \bar \Psi \trinr_2. \label{t1}
\end{align} 
A use of the fact that $\bar \Psi \in {\bH}^{2+\gamma}(\Omega)$, ${\mathcal B}(\bar \Psi) \in {\bH}^{-1}(\Omega)$ and $f \in H^{-1}(\Omega)$ along with  \eqref{sub} lead to an estimate for $T_2$ as 
\begin{align}
\trinl (T-T_h) ({\mathcal B}(\bar \Psi) - f)\trinr_2 \le Ch^{\gamma} (\trinl \bar \Psi \trinr_2^2  + \|f \|_{H^{-1} (\Omega)})), \label{t2}
\end{align}
where $\gamma \in (1/2,1]$ is the elliptic regularity index. 
Since $\bar u \in L^2(\omega)$, $T_3$ can also be estimated using \eqref{sub} as 
\begin{align} \label{t3}
\trinl (T-T_h) ({\mathcal C}{\bar u}) \trinr_2 \le Ch^{\gamma} \|\bar u\|_{L^2(\omega)}.
\end{align}
The boundedness of $T_h$ from Lemma \ref{bd_Th} leads to 
\begin{align} \label{t4}
\trinl T_h ({\mathcal C}{\bar u} - {\mathcal C}u) \trinr_2\leq C\trinl {\mathcal C}{\bar u} - {\mathcal C}u \trinr \le C \|\bar u -u\|_{L^2(\omega)}.
\end{align}
The substitutions of \eqref{t1}-\eqref{t4} in \eqref{main} yield 
$$
\trinl {\mathscr G}(\Psi,u)- {\bar \Psi} \trinr_2   \le \hat{C} (h^{\gamma} + \rho^2).$$
Choose  $\hat{\rho}_1 \le \min\left\{\rho_0, \frac{1}{2 \hat {C}} \right \}, \; \hat{\rho}_2= {\hat{\rho}}_1^2,$ and 
$\hat{h}_1 = \min \left \{ {h_0^{1/\gamma}, \left(\frac{\hat{\rho}_1}{2 \hat{C}} \right)^{1/\gamma} } \right\}$. For all $0 < h < {\hat h}_1$ and all $u \in B_{\hat \rho_2}(\bar u)$, ${\mathscr G}(\Psi,u)$  is a mapping from $B_{{\hat \rho}_1}(\bar \Psi) $ into itself.
\bigskip

\noindent {\it Step 2: } Let $\Psi_1, \Psi_2  \in B_{{\hat \rho}_1}(\bar \Psi), \; $ $0 < h < {\hat h}_1$ and $u \in B_{\hat \rho_2}(\bar u)$. The definition of the mapping $ {\mathscr G}(\cdot,u)$ and standard calculations lead to 
\begin{align}
&\trinl {\mathscr G}(\Psi_1,u) -  {\mathscr G}(\Psi_2,u) \trinr_2  = 
\trinl \Psi_1 - \Psi_2 -  [\partial_{\Psi}  {\mathscr N}_h({\bar \Psi}, {\bar u})]^{-1} 
\left\{{\mathscr N}_h(\Psi_1, u) - {\mathscr N}_h(\Psi_2, u)\right\}   \trinr_2 \nonumber 
\\
& = \trinl [\partial_{\Psi}  {\mathscr N}_h({\bar \Psi}, {\bar u})]^{-1}  
\left\{ \partial_{\Psi}  {\mathscr N}_h({\bar \Psi}, {\bar u}) (\Psi_1 - \Psi_2)- 
\int_0^1 \partial_{\Psi} {\mathscr N}_h(\Psi_2 + t(\Psi_1-\Psi_2), u)(\Psi_1-\Psi_2) \: dt  \right\}
\trinr_2.
\end{align} 
Now a use of Theorem \ref{invert} and a repetition of arguments used in \eqref{t1} lead to the result that, there exists a positive constant ${\tilde C}$ independent of ${\hat \rho}_1$ and $h$ such that 
\begin{align}
&\trinl {\mathscr G}(\Psi_1,u) -  {\mathscr G}(\Psi_2,u) \trinr_2 \le {\tilde C}{\hat \rho}_1^2.
 \end{align} 
A  choice of $\rho_1= \min\left \{ \rho_0,  \frac{1}{2 \hat {C}},  \frac{1}{\sqrt{2 \tilde {C}}} \right\}$, $\rho_2= \rho_1^2$, and $h_1 = \min \left \{ {h_0^{1/\gamma}, \left(\frac{\rho_1}{2 \hat{C}} \right)^{1/\gamma} } \right\}$ leads to the result that for all $0 < h< h_1$
and $u \in B_{\rho_2}(\bar u)$, $ {\mathscr G}(\cdot, u)$  is a strict contraction in $B_{\rho_1}({\bar \Psi}) $.

\medskip

\noindent This concludes the proof of Theorem \ref{aux_dis}. \hfill{$\Box$}

\bigskip 
 We have established that, for $0< h< h_1$, $u \in B_{\rho_2}({\bar u})$, ${\mathscr N}_h(\Psi_h,u)=0$ admits a unique solution $\Psi_{u,h} \in B_{\rho_1}({\bar \Psi}) \cap \bV_h$. Also, $\partial_{\Psi}  {\mathscr N}_h(\Psi_{u,h},u)$ is an automorphism in $\bV$. 
Hence, the mapping $G_h: B_{\rho_2}({\bar u}) \rightarrow B_{\rho_1}({\bar \Psi}) \cap \bV_h$ defined by $G_h(u)= \Psi_{u,h}$ satisfies ${\mathscr N}_h( G_h(u),u)=0$.  The implicit theorem yields that $G_h$ is of class $C^{\infty}$ in the interior of the ball.

\medskip
\noindent This fact, along with Theorem \ref{ee1} yields the following lemma. 
\begin{lem}\label{fourpointfive}
 For $u, \:{\hat u} \in B_{\rho_2}({\bar u}), \: 0 < h < h_1$,  the solutions $\Psi_{u}$ and $\Psi_{{\hat u},h}$ to \eqref{wv1} and \eqref{wform2}, with controls chosen as $u$ and ${\hat u}$ respectively, satisfy
 \begin{align*}
 \trinl \Psi_{u}- \Psi_{{\hat u},h}\trinr_2 \le C (h^{\gamma} + \|u- {\hat u}\|_{L^2(\omega)}), \; 
 \end{align*}
 $\gamma \in (1/2,1]$ being the elliptic regularity index.
\end{lem}
\begin{proof}
The triangle inequality yields 
\begin{align} \label{tr1}
\trinl \Psi_{u}- \Psi_{{\hat u},h}\trinr_2 \le  \trinl \Psi_{u}- \Psi_{u,h}\trinr_2 +
\trinl \Psi_{u,h}- \Psi_{{\hat u},h}\trinr_2.  
\end{align}
Theorem \ref{ee1} yields the estimate for the first term on the right hand side of \eqref{tr1}
as 
$$ \trinl \Psi_{u}- \Psi_{u,h}\trinr_2 \le Ch^{\gamma}.$$
\smallskip
\noindent From the expression  ${\mathscr N}_h(G_h(u),u) =0$ and the definition of ${\mathscr N}_h$, we obtain 
$$G'_h(u)(v)= -[\partial_{\Psi}{\mathscr N}_h(\Psi_{u,h}, u )]^{-1} T_h({\mathcal C}v),$$
where $u$ belongs to the interior of $B_{\rho_2}({\bf u})$.

\medskip
\noindent Hence  Lemma \ref{bd_Th} and Theorem \ref{invert} yield
\begin{align} 
\trinl \Psi_{u,h}- \Psi_{{\hat u},h}\trinr_2 & = \trinl  G_h(u) - G_h({\hat u}) \trinr_2 \nonumber \notag\\
& =  \trinl \int_{0}^1 [\partial_{\Psi}{\mathscr N}_h(\Psi_h({u_t}), u_t )]^{-1}
 T_h({\mathcal C}(u-{\hat u}))   \trinr_2 \nonumber \notag\\
 & \le C \|u - {\hat u} \|_{L^2(\omega)},\label{apriori_state}
\end{align}
with $u_t= {\hat u} + t(u- \hat{u})$. A substitution of the estimate in \eqref{tr1} yields the required result.

\end{proof}

\subsection{Auxiliary discrete problem for the adjoint equations}\label{sec3.3}
Define  an auxiliary continuous problem associated with the adjoint equations as follows:

\bigskip

\noindent Seek $\Theta_{u}\in \bV$ such that 
  \begin{align}
   A(\Phi, \Theta_{u})+B(\Psi_{u}, \Phi, \Theta_{u}) + B(\Phi, \Psi_{u}, \Theta_{u})=(\Psi_{u}- \Psi_d, \Phi) \; \fl  \Phi \in \bV,  
  \label{adj1}
  \end{align}
where  $\Psi_{u} \in \bV$ is the solution of \eqref{wv1} and $\Psi_d$ is given.
A conforming finite element discretization for \eqref{adj1} is defined as: 

Seek $\Theta_{u,h}\in \bV_h$ such that 
  \begin{align} 
 A(\Phi_h, \Theta_{u,h})+B(\Psi_{u,h}, \Phi_h, \Theta_{u,h}) + B(\Phi_h, \Psi_{u,h}, \Theta_{u,h})=(\Psi_{u,h}- \Psi_d, \Phi_h) \; \fl  \Phi_h \in \bV_h.
  \label{adj2}
  \end{align}
The main results  of this subsection will be on the existence of solution of the discrete adjoint problem in \eqref{adj2} and its error estimates. They are stated now.
\begin{thm} \label{aux_dis1}
Let $({\bar \Psi}, {\bar u}) \in \bV \times L^2(\omega)$ be a nonsingular solution of \eqref{of}. 
Then, there exist $0 < \rho_3 \le \rho_2$ and $h_3 >0$ such that, for all $0 < h \le h_3$ and ${u} \in B_{\rho_3}({\bar u})$, \eqref{adj2} admits a unique solution.
\end{thm}
\begin{thm} \label{ee2}
 Let $({\bar \Psi}, {\bar u}) \in \bV \times L^2(\omega)$ be a nonsingular solution of \eqref{of}.   Then, for ${u} \in B_{\rho_3}({\bar u})$ and $0 < h< h_3$, the solutions $\Theta_{u}$ and $\Theta_{u,h}$ of \eqref{adj1} and \eqref{adj2} satisfy the error estimates:
\begin{equation} \label{adjest}
(a) \; \trinl \Theta_{u}-\Theta_{u,h}\trinr_{2} \le C h^\gamma \qquad (b) \;  \; \trinl\Theta_{u}-\Theta_{u,h}\trinr_{1} \le C h^{2\gamma},
\end{equation} 
where $\gamma \in (1/2,1]$ denotes the index of elliptic regularity.
\end{thm}

\noindent{\bf A linear mapping and its properties:}

\medskip

As in the case of the derivative mapping defined in the previous subsection for state equations, define the linear mapping ${\mathscr F}_{\Psi} (\mbox{resp. } {\mathscr F}_{\Psi,h}  ) \in {\mathcal L}(\bV)$ by
$${\mathscr F}_{\Psi}(\Phi) = \Phi + T[{\mathcal B}'(\Psi)^* \Phi] \quad \forall \Phi \in \bV, $$
$$\left( {\rm resp.} \; {\mathscr F}_{\Psi,h}(\Phi) = \Phi + T_h[{\mathcal B}'(\Psi)^* \Phi] \quad \forall \Phi \in \bV \right)$$
where ${\mathcal B}'(\Psi)^*$ is the adjoint operator corresponding to  ${\mathcal B}'(\Psi)$ (see \eqref{opad}).

\noindent The next lemma is easy to establish and hence the proof is skipped.
\begin{lem}  The mapping ${\mathscr F}_{\Psi}$   is an automorphism in $\bV$ if and only if  $ \Psi \in \bV $ is a nonsingular solution of \eqref{wv}.
\end{lem}


\noindent{\it \textbf {Proof of Theorem \ref{aux_dis1}}:}

\medskip
Proceeding as in the proof of Theorem \ref{invert}, we can assume that $h_0$ is chosen so that, for all $0 < h< h_0$ and all ${\Psi} \in B_{\rho_0} ({\bar \Psi})$, ${\mathscr F}_{\Psi,h}$ is an automorphism in $\bV$. In particular, by  using Lemma \ref{fourpointfive}, there exist $0 < h_3 \le h_2$ and $0 < \rho_3 \le \rho_2$ such that, for all $0 < h \le h_3$ and all $u \in B_{\rho_3}({\bar u})$,  ${\mathscr F}_{\Psi_{u,h},h}$ is an automorphism in $\bV$ and 
$$ \|{\mathscr F}_{\Psi_{u,h},h}^{-1}  \|_{\cL(\bV)} \le 2 \| {\mathscr F}_{\bar \Psi}^{-1}\|_{\cL(\bV)}.$$

\noindent We can also assume that ${\mathscr F}_{\Psi_{u}}$ is an automorphism in $\bV$ for all 
$u \in B_{\rho_3}({\bar u})$.

\medskip
\noindent Now we establish that $\Theta_{u,h} \in \bV_h$ is a solution of \eqref{adj2} if and only if 
 ${\mathscr F}_{\Psi_{u,h},h}(\Theta_{u,h}) = {\boldsymbol \eta }_h $, where ${\boldsymbol \eta }_h \in \bV_h$ satisfies
$T_h(\Psi_{u} - \Psi_d) = {\boldsymbol \eta}_h$.  

\medskip
\noindent With definitions of ${\mathscr F}_{\Psi_{u,h},h}$ and the operator $T_h$, it yields 
 \begin{align*}
{\mathscr F}_{\Psi_{u,h},h}(\Theta_{u,h}) &= {\boldsymbol  \eta }_h 
\Longleftrightarrow \Theta_{u,h} + T_h[{\mathcal B}'(\Psi_{u,h})^* \Theta_{u,h}] =  {\boldsymbol \eta }_h  \nonumber \\
& \Longleftrightarrow A (\Theta_{u,h}, \Phi_h) + \big \langle{\mathcal B}'(\Psi_{u,h})^* \Theta_{u,h}, \Phi_h \big  \rangle_{\bV', \bV} = (\Psi_{u,h} - \Psi_d, \Phi_h) \; \forall \Phi_h \in \bV_h.
\end{align*}
That is, $\Theta_{u,h} \in \bV_h$ solves \eqref{adj2}. \hfill{$\Box$}

\bigskip
\noindent{\it \textbf {Proof of Theorem \ref{ee2}}:}

\medskip
The problem under consideration  being  linear, it is straight forward to obtain the required estimates. We will sketch the steps of the proof.

\medskip 
\noindent The space $\bV_h$ is a subspace of $\bV$ and hence  \eqref{adj1} holds true for test functions in $\bV_h$. 

\medskip

The definition of ${\mathscr F}_{\Psi_{u,h},h}$, and of the continuous and discrete adjoint problems lead to 
\begin{align*}
&{\mathscr F}_{\Psi_{u,h},h}(\Theta_{u}-\Theta_{u,h})={\mathscr F}_{\Psi_{u,h},h}(\Theta_{u})-{\mathscr F}_{\Psi_{u,h},h}(\Theta_{u,h})\\
&=\Theta_{u}+T_h[B'(\Psi_{u,h})^*\Theta_{u}]-T_h(\Psi_{u,h}-\Psi_d)\\
&=T(\Psi_{u}-\Psi_d)-T[B'(\Psi_{u})^*\Theta_{u}]+T_h[B'(\Psi_{u,h})^*\Theta_{u}]-T_h(\Psi_{u,h}-\Psi_d)\\
&=T(\Psi_{u}-\Psi_{u,h})+(T-T_h)(\Psi_{u,h}-\Psi_d-B'(\Psi_{u,h})^*\Theta_{u})+T[B'(\Psi_{u,h})^*\Theta_{u}-B'(\Psi_{u})^*\Theta_{u}].
\end{align*}
Since $F_{\Psi_{u,h},h}$ is an automorphism in $\bV_h$, the boundedness of $T$ leads to
\begin{align*}
\trinl\Theta_{u}-\Theta_{u,h}\trinr_2&\lesssim \trinl\Psi_{u}-\Psi_{u,h}\trinr_2+\|T-T_h\|\trinl \Psi_{u,h}-\Psi_d-B'(\Psi_{u,h})^*\Theta\trinr_2\\
&\quad+\|B'(\Psi_{u,h}-\Psi_{u})^*\Theta_{u}\|_{\bV'}.
\end{align*}
A use of Theorem~\ref{ee1}(a) and Lemma~\ref{BihErrEst} leads to the first estimate in \eqref{adjest}.
%

\bigskip 
\noindent To establish the second estimate in \eqref{adjest},  define an auxiliary problem and its discretization. 

\medskip
\noindent For all ${\bf g} \in {\bH}^{-1}(\Omega)$, let ${\boldsymbol \chi}_{\bf g} \in \bV$ and
 ${\boldsymbol \chi}_{{\bf g},h} \in \bV_h$ be the solutions to the equations 
 \begin{align}
 & A({\boldsymbol \chi}_{\bf g}, \Phi) + B(\Psi_{u}, {\boldsymbol \chi}_{\bf g}, \Phi) + B({\boldsymbol \chi}_{\bf g}, \Psi_{u},\Phi) = \langle {\bf g}, \Phi\rangle \quad \forall \Phi \in \bV.  \label{gg1} \\
&  A({\boldsymbol \chi}_{{\bf g},h}, \Phi_h) + B(\Psi_{u}, {\boldsymbol \chi}_{{\bf g},h}, \Phi_h) + B({\boldsymbol \chi}_{{\bf g},h}, \Psi_{u},\Phi_h) = \langle {\bf g}, \Phi_h\rangle \quad \forall \Phi_h \in \bV_h. \label{gg2}
 \end{align}
 The well-posedness of \eqref{gg1} implies that $\trinl{\boldsymbol \chi}_{\bf g}\trinr_2\lesssim \trinl{\bf g}\trinr_{-1}$ and  $\trinl{\boldsymbol \chi}_{\bf g}\trinr_{2+\gamma}\lesssim \trinl{\bf g}\trinr_{-1}$.
\noindent By proceeding as as in the proof of (a), we can establish that 
 \begin{align} \label{chi}
 \trinl {\boldsymbol \chi}_{\bf g} -  {\boldsymbol \chi}_{{\bf g},h} \trinr_2 \le C h^{\gamma}\trinl{\bf g}\trinr_{-1}, 
 \end{align}
 where the constant $C$ depends on $\trinl\Psi\trinr_2$, and $\gamma \in (1/2,1]$ is the index of elliptic regularity. 
 
\noindent
From \eqref{adj1} and \eqref{adj2}, it follows that 
\begin{equation} \label{adj3}
\begin{array}{l}
\displaystyle 
   A(\Phi_h, \Theta_{u} - \Theta_{u,h}) + B(\Psi_{u}, \Phi_h, \Theta_{u})
   +B(\Phi_h, \Psi_{u}, \Theta_{u})\vspace{2.mm}\\	
	-B(\Psi_{u,h}, \Phi_h,\Theta_{u,h}) 	
	- B(\Phi_h, \Psi_{u,h},\Theta_{u,h})
	=(\Psi_{u}- \Psi_{u,h}, \Phi_h) \quad  \mbox{for all}\ \Phi_h \in \bV_h.   
  \end{array}
\end{equation} 

\noindent  Choose $\Phi= \Theta_{u}-\Theta_{u,h}$ in \eqref{gg1} and adjustment of terms yield
\begin{align}
( {\bf g}, \Theta_{u}-\Theta_{u,h} ) & = A ({\boldsymbol \chi}_{\bf g}- {\boldsymbol \chi}_{{\bf g},h},  \Theta_{u}-\Theta_{u,h}) + B (\Psi_{u},{\boldsymbol \chi}_{\bf g}- {\boldsymbol \chi}_{{\bf g},h}  , \Theta_{u}-\Theta_{u,h}) \nonumber \\
& \; \; +B ({\boldsymbol \chi}_{\bf g}- {\boldsymbol \chi}_{{\bf g},h}  , \Psi_{u}, \Theta_{u}-\Theta_{u,h})+ B (\Psi_{u},{\boldsymbol \chi}_{{\bf g},h},\Theta_{u}-\Theta_{u,h})
\notag \\
& \; \; + B({\boldsymbol \chi}_{{\bf g},h},\Psi_{u}, \Theta_{u}-\Theta_{u,h})
+A ({\boldsymbol \chi}_{{\bf g},h},  \Theta_{u}-\Theta_{u,h}). \label{pre_gbdd}
\end{align}
Choose $\Phi_h={\boldsymbol \chi}_{{\bf g},h}$ in \eqref{adj3} and combine with \eqref{pre_gbdd} to obtain
 \begin{align}
  ( {\bf g}, \Theta_{u}-\Theta_{u,h} ) & = A ({\boldsymbol \chi}_{\bf g}- {\boldsymbol \chi}_{{\bf g},h},  \Theta_{u}-\Theta_{u,h}) + B (\Psi_{u},{\boldsymbol \chi}_{\bf g}- {\boldsymbol \chi}_{{\bf g},h}  , \Theta_{u}-\Theta_{u,h}) \nonumber \\
  & \quad +B ({\boldsymbol \chi}_{\bf g}- {\boldsymbol \chi}_{{\bf g},h}  , \Psi_{u}, \Theta_{u}-\Theta_{u,h})+ B(\Psi_{u,h} - \Psi_u, {\boldsymbol \chi}_{{\bf g},h}, \Theta_{u,h}  ) \notag\\
  &\quad+B( {\boldsymbol \chi}_{{\bf g},h},  \Psi_{u,h} - \Psi_u, \Theta_{u,h}  )+ (\Psi_{u}-\Psi_{u,h},{\boldsymbol \chi}_{{\bf g},h} )\notag \\
  & = A ({\boldsymbol \chi}_{\bf g}- {\boldsymbol \chi}_{{\bf g},h},  \Theta_{u}-\Theta_{u,h}) + B (\Psi_{u},{\boldsymbol \chi}_{\bf g}- {\boldsymbol \chi}_{{\bf g},h}  , \Theta_{u}-\Theta_{u,h}) \notag\\
 &\quad+B ({\boldsymbol \chi}_{\bf g}- {\boldsymbol \chi}_{{\bf g},h}  , \Psi_{u}, \Theta_{u}-\Theta_{u,h})+ B(\Psi_{u,h} - \Psi_u, {\boldsymbol \chi}_{{\bf g},h} -{\boldsymbol \chi}_{\bf g}, \Theta_{u,h}  )\notag\\
 &\quad+  B(\Psi_{u,h} - \Psi_u, {\boldsymbol \chi}_{\bf g}, \Theta_{u,h}  ) +
  B( {\boldsymbol \chi}_{{\bf g},h} -  {\boldsymbol \chi}_{\bf g},  \Psi_{u,h} - \Psi_u, \Theta_{u,h}  )
  \nonumber \\
  & \quad + B(  {\boldsymbol \chi}_{\bf g},  \Psi_{u,h} - \Psi_u, \Theta_{u,h}  ) + (\Psi_{u}-\Psi_{u,h},{\boldsymbol \chi}_{{\bf g},h} ).\label{g_bdd}
 \end{align}
 A choice of ${\bf g}= -\Delta (\Theta_{u}-\Theta_{u,h})$ in the above equation \eqref{g_bdd} and then integration by parts, and a use of boundedness properties \eqref{boundA}, \eqref{boundB1} and \eqref{boundB2} lead to 
 \begin{align}
 \trinl\Theta_{u}-\Theta_{u,h}\trinr_1^2&\lesssim \trinl{\boldsymbol \chi}_{\bf g}- {\boldsymbol \chi}_{{\bf g},h}\trinr_2\trinl  \Theta_{u}-\Theta_{u,h}\trinr_2+\trinl\Psi_{u,h} - \Psi_u\trinr_1\trinl{\boldsymbol \chi}_{\bf g}\trinr_{2+\gamma}+\trinl\Psi_{u,h} - \Psi_u\trinr\trinl{\boldsymbol \chi}_{{\bf g},h}\trinr\notag\\
 &\lesssim \trinl{\boldsymbol \chi}_{\bf g}- {\boldsymbol \chi}_{{\bf g},h}\trinr_2\trinl  \Theta_{u}-\Theta_{u,h}\trinr_2+\trinl\Psi_{u,h} - \Psi_u\trinr_1\trinl{\boldsymbol \chi}_{\bf g}\trinr_{2+\gamma}\notag\\
 &\quad+\trinl\Psi_{u,h} - \Psi_u\trinr\trinl{\boldsymbol \chi}_{\bf g}-{\boldsymbol \chi}_{{\bf g},h}\trinr+\trinl\Psi_{u,h} - \Psi_u\trinr\trinl{\boldsymbol \chi}_{\bf g}\trinr.
 \end{align}
 Note that $\trinl{\bf g}\trinr_{-1}=\trinl \Theta_{u}-\Theta_{u,h}\trinr_{1}$, and the well posedness of \eqref{gg1} implies $\trinl{\boldsymbol \chi}_{\bf g}\trinr_2\leq \trinl \Theta_{u}-\Theta_{u,h}\trinr_{1}$ and   $\trinl{\boldsymbol \chi}_{\bf g}\trinr_{2+\gamma}\leq \trinl \Theta_{u}-\Theta_{u,h}\trinr_{1}$. This, and estimates \eqref{chi}, part (a) of \eqref{adjest} and part (b) of \eqref{ee1_est} lead to part (b) estimate of \eqref{adjest}. \hfill{$\Box$}

 \bigskip
 
\noindent As for the case of the state equations (see Lemma \ref{fourpointfive}), we have the following result.
 \begin{lem}\label{fourpointfiveadj}
 For $u, \:{\hat u} \in B_{\rho_3}({\bar u}), \: 0 < h < h_3$, the solutions $\Theta_{u}$ and $\Theta_{{\hat u},h}$ to \eqref{adj1} and \eqref{adj2} with corresponding controls chosen as $u$ and ${\hat u}$ respectively, satisfy
 \begin{align*}
 \trinl \Theta_{u}- \Theta_{{\hat u},h}\trinr_2 \le C (h^{\gamma} + \|u- {\hat u}\|_{L^2(\omega)}), \; 
 \end{align*}
 $\gamma \in (1/2,1]$ being the elliptic regularity index.
\end{lem}

\section{Control discretization}\label{sec4}

First we describe the discretization of the control variable and then formulate the fully discrete problem.  This is followed by existence and convergence results for the discrete problem.
{
We make the following assumptions:
\begin{itemize}
	\item[{\bf (A1)}] Let $\omega\subset \Omega$ be a polygonal domain.
	\item[{\bf (A2)}] Assume that $\cT_h$ restricted to $\omega$ yields a triangulation for $\bar{\omega}$.
\end{itemize}
Note that the above assumptions are not very restrictive in practical situations. In case $\omega$ is not a polygonal domain, it can be approximated by a polygonal domain. The second assumption can be realized easily by choosing an initial triangulation appropriately. }

\medskip
\noindent Set
\begin{align*}
U_{h, ad}&=\left\{{{u}}\in L^2(\omega):  u|_T\in P_0(T), \; u_a \leq  u \leq
u_b \mbox{ for all }T\in
\mathcal{T}_h \right\}.
\end{align*}

\noindent The discrete control problem associated with \eqref{wform} is defined as follows:  
\begin{subequations}\label{discrete_cost}
	\begin{align} 
	&  \min_{(\Psi_h, {u}_h) \in  \bV_h \times  U_{h,ad}} J(\Psi_h, u_h) \,\,\, \textrm{ subject to } \\
	&  A(\Psi_h,\Phi_h)+B(\Psi_h,\Psi_h,\Phi_h)=(F + {\mathbf C}{\bf u}_h, \Phi_h),\label{dis_cost_b}
	\end{align}
\end{subequations}
for all $\Phi_h\in \bV_h$.

Recall that $(\Psi_h, u_h)$ satisfies \eqref{dis_cost_b} if
and only if
\begin{equation}
{\mathscr N}_h(  \Psi_h, u_h)=0.
\end{equation}

\noindent We aim to study the existence of local minima of \eqref{discrete_cost}
 which approximate the local minima of \eqref{wform}. This can be
established for nonsingular local solutions of \eqref{discrete_cost}.

\bigskip
\noindent The following lemma is crucial  in establishing the existence of solution of \eqref{discrete_cost} in Theorem  \ref{T4.11}.
\begin{lem} \label{aux_new}
 Let $({\bar \Psi}, {\bar u})\in\bV\times L^2(\omega)$ be a nonsingular solution of \eqref{wform}.   If ${u}_h \in B_{\rho_2}({\bar u})$ and $u_h \rightharpoonup u$ weakly, then  $\Psi_{u_h,h}$ converges to  $\Psi_{u}$ in ${\bH}^{2}(\Omega)$, where $\rho_2>0$ defined in the proof of Theorem \ref{aux_dis} denotes the radius of the ball $B_{\rho_2}(\bar{u})$ such that the discrete state equation \eqref{wform2} admits a unique solution, when the mesh parameter $h$, is chosen sufficiently small.
\end{lem}
\begin{proof}
Let $(u_h)_h$ be a sequence in $ B_{\rho_2}({\bar u}) \cap U_{ad}$ converging weakly to $u$. The result \eqref{app2}  in Lemma \ref{ap} yields that $\Psi_{u}$ and $\Psi_{u_h}$ belong to ${\bH}^{2+\gamma}(\Omega)$ and are bounded in  ${\bH}^{2+\gamma}(\Omega)$. Thus, there exists a subsequence (still denoted by the same notation) such that 
\begin{align*}
& \Psi_{u_h}  \rightharpoonup \tilde{\Psi} \mbox{  in } {\bH}^{2+\gamma}(\Omega), \\
  & \Psi_{u_h}  \rightarrow  \tilde{\Psi} \mbox{  in } {\bH}^{2}(\Omega).
\end{align*}
Note that $\Psi_{u_h}$ satisfies 
$$A (\Psi_{u_h}, \Phi) + B(\Psi_{u_h}, \Psi_{u_h}, \Phi) = \big \langle F + {\mathcal C}{\bf u}_h, \Phi \big \rangle_{\bV', \bV} \quad \forall \Phi \in \bV.$$
By passing to the limit, we have $\tilde{\Psi}=\Psi_{u_h}$. That is,  $\Psi_{u_h}  \rightarrow  \tilde\Psi_{u} \mbox{  in } {\bH}^{2}(\Omega)$. 

\medskip
\noindent Now a combination of this convergence result with Theorem \ref{ee1}, along with the triangle inequality and the fact that $u_h$ is bounded yield that $\Psi_{u_h,h}$ converges to $\Psi_{u}$ in ${\bH}^2(\Omega)$.
\end{proof}

\begin{cor} A result analogous to Lemma \ref{aux_new} holds true for the convergence of the solutions of the continuous and discrete adjoint problems as well. That is, for a nonsingular solution $({\bar \Psi}, {\bar u}) \in \bV \times L^2(\omega)$ of \eqref{of}, if ${u}_h \in B_{\rho_3}({\bar u})$ and $u_h \rightharpoonup u$ weakly in $L^2(\omega)$, then  $\Theta_h(u_h)$ converges to $ \Theta_{u}$ in ${\bH}^{2}(\Omega)$, where $\rho_3 >0$ is as defined in Theorem \ref{aux_dis1}.
\end{cor}
The next theorem states the existence of at least one solution of the discrete control problem stated in \eqref{discrete_cost} and the convergence results for the control and state variables. Since the proof is quite standard (for example, see the proof of Theorem 4.11 in \cite{cmj}), it is skipped.  
\begin{thm}\label{T4.11}
Let $({\bar \Psi}, {\bar u}) \in \bV \times L^2(\omega) $ be a nonsingular solution of \eqref{wform}.   Then there exists $h_2
> 0$ such that, for all $0<h < h_2$, \eqref{discrete_cost} has at least one solution. If furthermore
$({\bar \Psi}, {\bar u})$ is a strict local minimum of \eqref{wform}, then for all $0<h < h_2$, \eqref{discrete_cost} has a local minimum $({\bar \Psi}_h, {\bar u}_h)$ in a neighborhood of
$({\bar \Psi}, {\bar u})$  and the following results hold:
\[
\lim_{h\to 0} j_h(\bar u_h)=j(\bar u), \
\lim_{h\to
0}\|\bar u -\bar u_h  \|_{L^2(\omega)} = 0
\mbox{ and } \lim_{h\to
0} \trinl \bar{\Psi} -\bar{\Psi}_h \trinr_2 = 0,
\]
where $j_h(\bar u_h) = {\mathscr N}( {\bar\Psi}_h, {\bar u}_h)$.
\end{thm}

Let  $({\bar \Psi}, {\bar u})$  be a nonsingular
strict local minimum of \eqref{wform} and  $\{({\bar \Psi}_h, {\bar u}_h)  \}_{h \le
h_3}$ be a sequence of local minima of problems \eqref{discrete_cost} converging
to $({\bar \Psi}, {\bar u})$ in $\bV \times L^2(\omega)$ , with ${\bar u}_h \in B_{\rho_3}(\bar u)$, where
$h_3$ and $\rho_3$ are given by Theorem \ref{aux_dis1}. Then every element ${\bar u}_h$  from a sequence
$\{\bar u_h\}_{h \le h_3}$ is a local solution of the problem with a discrete reduced cost functional 
\begin{align} \label{reduced_discrete}
\min_{{{u}}\in  U_{h,ad}} j_h({{u}}) := {\mathscr N}_h(\Psi_{u,h}, u)
\end{align}
where $\Psi_{u,h}= G_h( u )$.

\medskip
\noindent In the next lemma, we establish the optimality condition  for the discrete control problem and the uniform convergence of the controls.
\begin{lem}
 Let $\bar{u}_h$ be a solution to problem
\eqref{reduced_discrete}, and let
${\bar \Psi}_h, \: {\bar \Theta}_h \in \bV_h$ denote the
corresponding discrete state and adjoint state. Then ${\bf \bar u}_h=({\bar u}_h,0)$
satisfies
\begin{align}\label{dis_opt}
\int_{\omega} ({\mathbf {C}}^* {\bar \Theta}_h +\alpha \bar{\bf u}_h ) 
\cdot ({{ \bf u}}_h-\bar{\mathbf{u}}_h)  \, dx \geq 0\ \mathrm{ for\ all\
}{\mathbf{u}}_h  = (u_h,0) ,   \; u_h \in  U_{h,ad}.
\end{align}
Also, $\displaystyle \lim_{h \rightarrow 0} \| {\bar u} - {\bar u}_h \|_{L^{\infty}(\omega)}=0$.
\end{lem}
\begin{proof}
For the first part, we use the optimality condition for the reduced discrete cost functional. That is,
$$ j_h'({{\bar u}_h})(u_h-{\bar u}_h) = \int_{\omega} ({\rm C}^* {\bar \theta}_{1h} + \alpha {\bar u}_h)(u_h - {\bar u}_h) \dx \ge 0,$$
from which the required result \eqref{dis_opt} follows, as ${\bf \bar u}_h=({\bar u}_h,0)$,  $ {\mathbf{u}}_h  = (u_h,0)$ and ${\bar \Theta}_h=({\bar \theta}_{1h}, {\bar \theta}_{2h})$.

\medskip
\noindent From \eqref{dis_opt}, we can express the discrete control as the projection of the adjoint variable on $[u_a,u_b]$. 
That is, 
$$ {\bar u}_h|_T = \pi_{[u_a,u_b]}\left(  -\frac{1}{\alpha |T|} \int_T ({\rm C}^* \theta_{1h} )(x) \dx\right). $$
For $x \in T$, the projection formula for the continuous control in \eqref{rep}, the mean value theorem and the Lipschitz continuity of the projection operator yield 
\begin{align*}
|{\bar u}_h(x)- {\bar u}(x)| & \le \left| \frac{1}{\alpha |T|} \int_T  ({\rm C}^* \theta_{1h} )(s) \: ds  - \frac{1}{\alpha} ({\rm C}^* {\bar \theta}_{1h} )(x)\right| \\
&= \frac{1}{\alpha} \left|   ({\rm C}^* {\bar \theta}_{1h} )(x_T) -  ({\rm C}^* {\bar \theta}_{1} )(x) \right| \\
& \le \left|   ({\rm C}^* {\bar \theta}_{1h} )(x_T) -  ({\rm C}^* {\bar \theta}_{1} )(x_T) \right| + 
\left|   ({\rm C}^* {\bar \theta}_{1} )(x_T) -  ({\rm C}^* {\bar \theta}_{1} )(x) \right| \\
& \le C \left(\trinl {\bar \Theta}_h - {\bar \Theta} \trinr_{\infty} + |x_T-x|\right) \\
& \le C \left( \trinl {\bar \Theta}_h - {\bar \Theta} \trinr_{\infty} +h \right),
\end{align*}
for some $x_T \in T$, and the result follows from the Sobolev imbedding result together with Lemma~\ref{fourpointfiveadj} and Theorem~\ref{T4.11}.
\end{proof}

\section{Error Estimates}

In this section, we develop error estimates for the state, adjoint and control variables. 

\medskip
Let  $({\bar \Psi}, {\bar u})$  be a nonsingular
strict local minimum of \eqref{wform} satisfying the second order optimality condition in Theorem \ref{ssoc} (or equivalently \eqref{ssoc2}). Let   $\{({\bar \Psi}_h, {\bar u}_h)  \}_{h \le
h_3}$ be a sequence of local minima of problems \eqref{discrete_cost} converging
to $({\bar \Psi}, {\bar u})$ in $\bV \times L^2(\omega)$ , with ${\bar u}_h \in B_{\rho_3}(\bar u)$, where
$h_3$ and $\rho_3$ are given by Theorem \ref{aux_dis1}. Since $h \le h_3$ and ${\bar u}_h \in B_{\rho_3}({\bar u})$, ${\bar u}_h$ is a local minimum of \eqref{reduced_discrete}. 

\medskip

First we state a lemma which is essential for the proof of the main convergence result in Theorem \ref{conv}. For a proof see \cite[Lemmas 4.16 \& 4.17]{cmj}.

\begin{lem}\label{aux_lemma}
(a) Let the second order optimality condition \eqref{ssoc2} hold true. Then, there exists a mesh size $h_4$ with
$0 < h_4 \le h_3$ such that 
\begin{align} \label{ssfo}
\frac{\delta}{2} \|{\bar u}-{\bar u}_h\|^2_{L^2(\omega)} \le \left( j'({\bar u}) -j'({\bar u}_h)\right)({\bar u}-{\bar u}_h) \; \; \forall\, 0 < h <h_4.
\end{align}
(b) There exist a mesh size $h_5$ with
$0 < h_5 \le h_4$ and a constant $C>0$ such that, for every $0 < h \le h_5$, there exists $u_h^* \in U_{h,ad}$ satisfying 
\begin{align} \label{uhstar}
(i) \; j'({\bar u})({\bar u}-u_h^*) =0 \qquad (ii) \; \|{\bar u}-u_h^*\|_{L^{\infty}(\omega)} \le Ch.
\end{align} 
\end{lem}

\noindent The following theorem establishes the convergence rates for control, state and adjoint variables. 
\begin{thm} \label{conv}
Let  $({\bar \Psi}, {\bar u})$  be a nonsingular
strict local minimum of \eqref{wform}  and  $\{({\bar \Psi}_h, {\bar u}_h)  \}_{h \le
h_3}$ be a solution to  \eqref{discrete_cost} converging
to $({\bar \Psi}, {\bar u})$ in $\bV \times L^2(\omega)$, for a sufficiently small mesh-size $h$ with ${\bar u}_h \in B_{\rho_3}(\bar u)$, where $\rho_3$ is given in Theorem \ref{aux_dis1}. Let $\bar \Theta $ and ${\bar \Theta}_h $ be the corresponding continuous and discrete adjoint state variables, respectively.  Then, 
there exists a constant $C>0$ such that, for all $0<h\leq h_5$, we have
\begin{align*}
(i) \; \|\bar{{u}}-\bar{{u}}_h\|_{L^2(\omega)}\leq C h  \quad (ii)\; \trinl \bar \Psi- {\bar \Psi}_h\trinr_2 \le Ch^{\gamma} \quad (iii) \; \trinl \bar \Theta- {\bar \Theta}_h\trinr_2 \le Ch^{\gamma},
\end{align*}
$\gamma \in (1/2,1]$ being the index of elliptic regularity.

\end{thm}

\begin{proof} 

For $0<h\leq h_5$, from \eqref{ssfo}, we have 
\begin{equation}\label{E5.100}
\begin{array}{l}
\displaystyle{
\frac{\delta}{2}\|\bar{{u}}-\bar{{u}}_h\|^2_{L^2(\omega)}\leq
(j'(\bar{{u}})-j'( \bar u_h ))(\bar{{u}}- \bar u_h
) }
\vspace{2.mm}\\
\displaystyle{ = (j'(\bar{{u}}) -
j_h'(\bar u_h))(\bar{{u}}- \bar u_h )+ (j'_h(\bar u_h) - j'( \bar u_h ))(\bar u-
\bar u_h ). }
\end{array}
\end{equation}
We now proceed to estimate the two terms in the right hand side of \eqref{E5.100}. From first order optimality conditions for continuous and discrete problems,  we have
\begin{eqnarray*}
&
j'(\bar u)( \bar u_h -\bar u)\geq 0,\quad j'_h({\bar u}_h)( u_h^* -\bar u_h)\geq 0.
\end{eqnarray*}
Also, $0\leq \displaystyle j'_h({\bar u}_h)( u_h^* -\bar u_h)
= j'_h(\bar  u_h)( u_h^* -\bar u  )+
j'_h(\bar u_h)(\bar u -\bar u_h)
$ holds.

\noindent For $\bar{u}_h\in B_{\rho_3}(\bar{u})$, the above expressions, \eqref{adjest}, stability of the continuous adjoint solution and \eqref{uhstar} lead to
\begin{align}
&j'(\bar u)(\bar u -  \bar u_h ) - j'_h(\bar u_h)(\bar u  -\bar u_h)
\leq j'_h(\bar u_h)( u_h^* -\bar u  )\notag\\
&= j'_h(\bar u_h)( u_h^* -\bar u  ) - j'(\bar u) (  u_h^* -\bar u )\notag\\
&=\int_{\omega}({\mathcal {C}^*}(\bar \theta_{1h}-\bar\theta_{1})+ \alpha(\bar{{u}}_h-
\bar{{u}})) ({{u}}_h^*-\bar{{u}}) \dx\notag\\
&\leq C\left(\trinl \bar \Theta_h-\bar\Theta \trinr +
\|\bar{{u}}_h-\bar{{u}}\|_{L^2(\omega)} \right)
 \|{{u}}_h^*-\bar{{u}}\|_{L^2(\omega)}\notag\\
 &\leq Ch \left(\trinl \bar\Theta_h- \Theta_{{\bar u}_h} \trinr
  +\trinl  \Theta_{{\bar u}_h} - \bar \Theta \trinr +
\|\bar u_h-\bar u\|_{L^2(\omega)}\right) \nonumber \\
&\leq Ch\left( h^{2 \gamma}  +  \|\bar u_h-\bar u\|_{L^2(\omega)}\right).\label{E5.102}
\end{align}
The estimate \eqref{adjest} yields
\begin{align}\label{E5.101}
\displaystyle{ (j'_h(\bar u_h) - j'( \bar u_h ))(\bar u-
\bar u_h ) } &
\displaystyle{ = \int_{\omega} ({\mathcal C}^*(\bar \theta_{1h}-\theta_1({\bar u}_h)+ \alpha (\bar u_h-\bar u_h)) (\bar u -\bar u_h) \dx }
\nonumber \\
 & \leq C \trinl \bar \Theta_h-  \Theta_{{\bar u}_h} \trinr   \|\bar{{u}}-
\bar{{u}}_h\|_{L^2(\omega)}
\nonumber \\
& 
 \leq Ch^{2 \gamma}  \|\bar u-\bar u_h\|_{L^2(\omega)}.
\end{align}

\noindent A substitution of the expression \eqref{E5.102}-\eqref{E5.101} in \eqref{E5.100} along with the Young's inequality yields the first required estimate.

\medskip

A use of the control estimate $(i)$ in Lemmas \ref{fourpointfive} and \ref{fourpointfiveadj} yield the required estimates for state and adjoint variables in $(ii)$ and $(iii)$ respectively. This concludes the proof.
\end{proof}


%
%
\subsection{Post processing for control}

A post processing of control helps us obtain improved error estimates for control. Also,  error estimates for the state and adjoint variables in $H^1$ and $L^2$ norms are derived.  {Recall the assumptions {\bf (A1)} and {\bf (A2)} on $\omega$ and $\cT_h$ as described in Section~\ref{sec4}. }

\begin{defn}[Interpolant] Define the projection $\cP_h: C(\bar\Omega)\to U_{h,ad}$ by
\begin{equation*}
(\cP_h\chi)(x)=\chi(S_i)\quad\fl x\in T_i\in\cT_h,
\end{equation*}
where $S_i$ denotes the centroid of the triangle $T_i$.
\end{defn}

\begin{defn}[Post processed control] The post processed control $\widetilde{\bar{u}}_h$ is defined as:
\begin{equation}\label{defn_post_proc}
\widetilde{\bar{u}}_h=\Pi_{[u_a,u_b]}\left(-\frac{1}{\alpha}(C^*\bar\Theta_h)(x)\right),
\end{equation}
where $\bar\Theta_h$ is the discrete adjoint variable corresponding to the control $\bar{u}_h$.
\end{defn}
{
Let $\cT_h^1={\cT}_h^{1,1}\cup {\cT}_h^{1,2} $ denote the union of active and inactive set of triangles contained in $\omega$}, where $\bar{u}(x)$ satisfies 
\begin{align*}
&\bar{u}  \equiv u_a \; \mbox { on } T; \;
\: \bar{u}  \equiv u_b \; \mbox { on } T \;\;(\text{in the active part } \cT_h^{1,1}),\\
&u_a < \bar{u} < u_b \; \mbox { on } T \;\;  (\text{in the inactive part } \cT_h^{1,2}),
\end{align*}
 and ${\cT}_h^2 := {\cT}_h \setminus {\cT}_h^1$, the set of triangles, where $\bar{u}$ takes on the values $u_a $ (resp.  $u_b$) as well as values
 greater that $u_a$ (resp.  lesser than $u_b$).
Let $\Omega_h^1= int \left( \cup_{T \in {\cT}_h^1} T\right)$ (where the notation $int$ denotes the interior) be the uncritical part and let $\Omega_h^{1,1}$ and $\Omega_h^{1,2}$
be the union of the triangles in the active  and inactive parts,
respectively. That is, $\Omega_h^1 = int \left(\overline{\Omega_h^{1,1}
\cup \Omega_h^{1,2}} \right)$ with $\Omega_h^{1,1}= int \left( \cup_{T
\in {\cT}_h^{1,1}} T\right)$, $\Omega_h^{1,2}= int \left(
\cup_{T \in {\cT}_h^{1,2}} T\right)$. Define $\Omega_h^2= int
\left( \cup_{T \in {\cT}_h^2} T\right)$ as the critical part of
${\cT}_h$.
We make an assumption on $\Omega_h^2$, the set of critical triangles which is fulfilled in practical cases \cite{FreiRannacherWollner13}:
\begin{itemize}
	\item[{\bf (A3)}] Assume $\displaystyle|\Omega_h^2| = \sum_{T \in {\cT}_h^2 } |T| < Ch$, for some positive constant $C$ independent of $h$.
\end{itemize}
  \begin{equation} \label{assumption}
   \qquad |\Omega_h^2| = \sum_{T \in {\cT}_h^2 } |T| < Ch,
  \end{equation}
for some positive constant $C$ independent of $h$.  
This implies that  the mesh domain of the critical cells is sufficiently small.

\noindent Use the splitting $ \bar{\Omega} =\bar{\Omega}_h^1 \cup \bar{\Omega}_h^2$,
 to define a discrete norm  $\|\cdot\|_{\bar h}$ for the control as
 $$ \|\bar{u}\|_{{\bar h}} := \trinl \bar{u}\trinr_{H^2(\Omega_h^1)} +  \trinl \bar{u}\trinr_{W^{1, \infty}(\Omega_h^2)},
$$
where $\displaystyle\trinl \bar{u}\trinr_{H^2(\Omega_h^1)}^2:= \sum_{T\in \Omega_h^1}\|\bar{u}\|_{H^2(T)}^2$ and $\displaystyle \trinl \bar{u}\trinr_{W^{1, \infty}(\Omega_h^2)}^2:= \sum_{T\in \Omega_h^2}\|\bar{u}\|_{W^{1, \infty}(T)}^2$.

\begin{lem}[Numerical integration estimate] \cite[Lemma 3.2]{MeyerRosh04} Let $g$ be a function belonging to $H^2(T_i)$ for all $i$
in a certain index set $I$. Then, there holds :
\begin{eqnarray} \label{ni}
\left|\int_{T_i} (g(x) - g(S_i)) \: dx \right| & \le &  Ch^2 \sqrt{|T_i|} \: |g|_{H^2(T_i)}
\end{eqnarray}

where $S_i$ denotes the centroid of $T_i$.
\end{lem}
Also, the following result can be established using scaling
arguments. 
\begin{lem}[Scaling results] For $\bar{u} \in W^{1, \infty}(T_i) $, (resp. $H^2(T_i)$) with $T_i \in \cT_h$,
\begin{equation} \label{scale_1}
\|\bar{u}- {\cP}_h\bar{u} \|_{L^{\infty}(T_i)} \le Ch \|\bar{u} \|_{W^{1,
\infty}(T_i)},\: (\text{resp. }\|\bar{u}- {\cP}_h\bar{u}
\|_{0,T_i} \le Ch \|\bar{u} \|_{H^2(T_i)}).
\end{equation}
\end{lem}

\begin{thm}\label{PP_aux}
Let $\Psi_{\bar{u},h}$ and $\Psi_{\cP_h\bar{u},h}$ be solutions of \eqref{wform2} with respect to control $\bar{u}$ and post processed control $\cP_h\bar{u}$, respectively. Then the following error estimate holds true:
\begin{equation*}
\trinl \Psi_{\bar{u},h}-\Psi_{\cP_h\bar{u},h}\trinr\leq C h^2\|\bar{u}\|_{\bar{h}}.
\end{equation*}
\end{thm}  

\begin{proof}
Consider the perturbed auxiliary problem:

 Seek $\bxi\in\bV$ that solves
\begin{equation}\label{pert_aux}
A(\bz,\bxi)+B(\Psi_{\bar{u},h},\bz,\bxi)+B(\bz,\Psi_{\cP_h\bar{u},h},\bxi)=(\Psi_{\bar{u},h}-\Psi_{\cP_h\bar{u},h},\bz)\quad\forall \bz\in\bV.
\end{equation}
Its discretization is given by: 

Seek $\bxi_h\in\bV_h$ that solves
\begin{equation}\label{pert_aux_dis}
A(\bz_h,\bxi_h)+B(\Psi_{\bar{u},h},\bz_h,\bxi_h)+B(\bz_h,\Psi_{\cP_h\bar{u},h},\bxi_h)=(\Psi_{\bar{u},h}-\Psi_{\cP_h\bar{u},h},\bz_h)\quad\forall \bz_h\in\bV_h.
\end{equation}
The above equation \eqref{pert_aux_dis} can be written in the operator form as 
\begin{equation}
\cA^*\bxi_h+\cB'(\Psi_{\bar{u},h})^*\bxi_h+\half\cB'(\Psi_{\cP_h\bar{u},h}-\Psi_{\bar{u},h})^*\bxi_h=T_h(\Psi_{\bar{u},h}-\Psi_{\cP_h\bar{u},h})\text{ in }\bV_h.
\end{equation}
Note that \eqref{apriori_state} and \eqref{scale_1} lead to 
\begin{equation}
\trinl\Psi_{\cP_h\bar{u},h}-\Psi_{\bar{u},h}\trinr_2\leq C \|\cP_h\bar{u}-\bar{u}\|_{L^2(\omega)}\leq Ch. \label{temp10}
\end{equation}
The invertibility of $\cA^*+\cB'(\Psi_{\bar{u},h})^*$, Lemma~\ref{banach} and \eqref{temp10} lead to well-posedness of \eqref{pert_aux_dis}.
Choose $\bz_h=\Psi_{\bar{u},h}-\Psi_{\cP_h\bar{u},h}$ in \eqref{pert_aux_dis} and simplify the terms to obtain
\begin{align}
\trinl\Psi_{\bar{u},h}-\Psi_{\cP_h\bar{u},h}\trinr^2&=A(\Psi_{\bar{u},h}-\Psi_{\cP_h\bar{u},h},\bxi_h)+B(\Psi_{\bar{u},h},\Psi_{\bar{u},h},\bxi_h)\notag\\
&\qquad-B(\Psi_{\cP_h\bar{u},h},\Psi_{\cP_h\bar{u},h},\bxi_h).\label{L2app}
\end{align}
Note that $\Psi_{\bar{u},h}$ and $\Psi_{\cP_h\bar{u},h}$ satisfy the following discrete problems:
\begin{align*}
&A(\Psi_{\bar{u},h},\Phi_h)+B(\Psi_{\bar{u},h},\Psi_{\bar{u},h},\Phi_h)=(F+\cC\bar{u},\Phi_h)\quad\forall \Phi_h\in \bV_h,\\
&A(\Psi_{\cP_h\bar{u},h},\Phi_h)+B(\Psi_{\cP_h\bar{u},h},\Psi_{\cP_h\bar{u},h},\Phi_h)=(F+\cC(\cP_h\bar{u}),\Phi_h)\quad\forall \Phi_h\in \bV_h.
\end{align*}
Subtract the above two equations to obtain
\begin{align*}
A(\Psi_{\bar{u},h}-\Psi_{\cP_h\bar{u},h},\Phi_h)+B(\Psi_{\bar{u},h},\Psi_{\bar{u},h},\Phi_h)-B(\Psi_{\cP_h\bar{u},h},\Psi_{\cP_h\bar{u},h},\Phi_h)=(\cC(\bar{u}-\cP_h\bar{u}),\Phi_h).
\end{align*}
Choose $\Phi_h=\bxi_h$ in the above equation and use \eqref{L2app} to obtain
\begin{equation}\label{pert_est}
\trinl\Psi_{\bar{u},h}-\Psi_{\cP_h\bar{u},h}\trinr^2=(\cC(\bar{u}-\cP_h\bar{u}),\bxi_h).
\end{equation}
Consider
\begin{eqnarray*}
\int_{\Omega_h^1} (\bar{u} - {\cP}_h \bar{u})\boldsymbol{\xi}_h \; dx & = &
\sum_{T_i \in {\Omega}_h^1} \left((\bar{u}- {\cP}_h \bar{u},
\boldsymbol{\xi}_h(S_i))_{T_i} + (\bar{u}- {\cP}_h \bar{u},
\boldsymbol{\xi}_h-\boldsymbol{\xi}_h(S_i))_{T_i}\right).
\end{eqnarray*}
A use of \eqref{ni} along with the result $|\boldsymbol{\xi}_h(S_i)| \sqrt{T_i} = \|\boldsymbol{\xi}_h(S_i)\|_{0,T_i} \le \|\boldsymbol{\xi}_h\|_{0,T_i}$ for the first term leads to
\begin{eqnarray}
\int_{\Omega_h^1}   (\bar{u} - {\cP}_h \bar{u})\boldsymbol{\xi}_h  \; dx& \le &
Ch^2 \sum_{T_i \in {\Omega}_h^1} \|\bar{u}\|_{H^2(T_i)}
\left( \|\boldsymbol{\xi}_h\|_{0,T_i}  + \|\boldsymbol{\xi}_h\|_{H^2(T_i)} \right) \nonumber \\
& \le & Ch^2 \trinl \bar{u}\trinr_{H^2(\Omega_h^1)}
\trinl \bxi_h\trinr_{H^2(\Omega_h^1)}. \label{active1}
\end{eqnarray}
Also, consider
\begin{eqnarray}
& & \int_{\Omega_h^2}   (\bar{u} - \cP_h \bar{u})\boldsymbol{\xi}_h \; dx =
\sum_{T_i \in \Omega_h^2}( \bar{u} - \cP_h \bar{u},\boldsymbol{\xi}_h)_{T_i} \nonumber \\
&  & \qquad  \le \|\bar{u} - \cP_h \bar{u} \|_{L^{\infty}
(\Omega_h^2)}
 \|\boldsymbol{\xi}_h\|_{L^{\infty} (\Omega_h^2)}
 \sum_{T_i \in \cT_h^2} |T_i|. \nonumber
\end{eqnarray} 
The assumption {\bf (A3)}, the estimate \eqref{scale_1} and Sobolev imbedding result in the above equation lead to
\begin{eqnarray}\label{active2}
& & \int_{\Omega_h^2} (\bar{u} - \cP_h \bar{u})\boldsymbol{\xi}_h \; dx \le C h^2 \trinl \bar{u} \trinr_{W^{1, \infty}(\Omega_h^2)} \trinl
\bxi_h\trinr_{H^2(\Omega_h^2)}.
\end{eqnarray}
{
A combination of \eqref{active1} and \eqref{active2} yields
\begin{align}\label{term1_est}
\int_{\Omega_h^1\cup\Omega_h^2 } (\bar{u} - \cP_h \bar{u})\boldsymbol{\xi}_h \; dx\le C h^2 \|\bar{u} \|_{\bar{h}} \trinl
\bxi_h\trinr_{2}.
\end{align} }
Now we estimate $\trinl\bxi_h\trinr_2$. Let $\cL$ (resp. $\cL_h$) $:\bV\to \bV$
be defined by 
\begin{align*}
\cL(\boldsymbol{\chi})&:=\boldsymbol{\chi}+\half T[B'(\Psi_{\bar{u},h})^*\boldsymbol{\chi}]+\half T[B'(\Psi_{\cP_h\bar{u},h})^*\boldsymbol{\chi}]\\
(\text{resp. } \cL_{h}(\boldsymbol{\chi})&:=\boldsymbol{\chi}+\half T_h[B'(\Psi_{\bar{u},h})^*\boldsymbol{\chi}]+\half T_h[B'(\Psi_{\cP_h\bar{u},h})^*\boldsymbol{\chi}]).
\end{align*}
The auxiliary perturbed problem and its discretization \eqref{pert_aux}-\eqref{pert_aux_dis} can now be expressed as
\begin{align*}
\cL(\bxi)&=T(\Psi_{\bar{u},h}-\Psi_{\cP_h\bar{u},h}),\\
\cL_{h}(\bxi_h)&=T_h(\Psi_{\bar{u},h}-\Psi_{\cP_h\bar{u},h}).
\end{align*}
From the above characterization, it follows that 
\begin{align*}
&\cL_h(\bxi-\bxi_h)=\cL_h(\bxi)-\cL_h(\bxi_h)\\
&=\bxi+\half T_h[B'(\Psi_{\bar{u},h})^*\boldsymbol{\xi}]+\half T_h[B'(\Psi_{\cP_h\bar{u},h})^*\boldsymbol{\xi}]-T_h(\Psi_{\bar{u},h}-\Psi_{\cP_h\bar{u},h})\\
&=(T-T_h)(\Psi_{\bar{u},h}-\Psi_{\cP_h\bar{u},h})-\half(T-T_h)[B'(\Psi_{\bar{u},h})^*\bxi]-\half(T-T_h)[B'(\Psi_{\cP_h\bar{u},h})^*\bxi].
\end{align*}
The invertibility of $\cL_h$ and Lemma~\ref{BihErrEst} lead to
\begin{align}
\trinl\bxi-\bxi_h\trinr_2&
\lesssim h^\gamma\left(\trinl\Psi_{\bar{u},h}-\Psi_{\cP_h\bar{u},h}\trinr+\trinl \bxi\trinr_{2+\gamma}\right).\label{errest_aux}
\end{align}
{Combine \eqref{term1_est} and \eqref{errest_aux}, and use triangle inequality together with the estimate for $\trinl\bxi\trinr_2$ and $\trinl\bxi\trinr_{2+\gamma}$ to obtain
\begin{equation}
(C(\bar{u}-\cP_h\bar{u}),\boldsymbol{\xi}_h)=\int_{\Omega_h^1\cup\Omega_h^2 } (\bar{u} - \cP_h \bar{u})\boldsymbol{\xi}_h \; dx\leq C h^2 \|\bar{u} \|_{\bar{h}} \trinl\Psi_{\bar{u},h}-\Psi_{\cP_h\bar{u},h}\trinr.
\end{equation} }
This and \eqref{pert_est} lead to the required estimate
\begin{equation*}
\trinl\Psi_{\bar{u},h}-\Psi_{\cP_h\bar{u},h}\trinr\leq Ch^2 \|\bar{u} \|_{\bar{h}}.
\end{equation*}
\end{proof}

Following the proof of the above theorem, the next result holds immediately.
\begin{cor}\label{uh_control}
Let $\Psi_{\bar{u}_h,h}$ and $\Psi_{\cP_h\bar{u},h}$ be the solutions of \eqref{wform2} with the control $\bar{u}_h$ and the post processed control $\cP_h\bar{u}$, respectively. Then the following error estimate holds true:
\begin{equation*}
\trinl \Psi_{\bar{u}_h,h}-\Psi_{\cP_h\bar{u},h}\trinr\leq C h^2\|\bar{u}\|_{\bar{h}}.
\end{equation*}
\end{cor}

The discrete post processed adjoint problem can be stated as: 

$\quad$Find $\Theta_{\cP_h\bar{u},h}\in\bV_h$ such that
\begin{equation}\label{dis_post_adj}
\tilde{\cL}_h(\Phi_h,\Theta_{\cP_h\bar{u},h}):=A(\Phi_h,\Theta_{\cP_h\bar{u},h})+2B(\Psi_{\bar{u},h},\Phi_h,\Theta_{\cP_h\bar{u},h})=(\Psi_{\cP_h\bar{u},h}-\Psi_d,\Phi_h)\quad\forall \Phi_h\in\bV_h.
\end{equation}

\begin{lem}\label{adj_post}
Let $\Theta_{\bar{u},h}$ be solution of \eqref{adj2} with the control $\bar{u}$ and $\Theta_{\cP_h\bar{u},h}$ be the solution of \eqref{dis_post_adj}. Then the following error estimate holds true:
\begin{equation*}
\trinl \Theta_{\bar{u},h}-\Theta_{\cP_h\bar{u},h}\trinr\leq C h^2\|\bar{u}\|_{\bar{h}}.
\end{equation*}
\end{lem}
\begin{proof}
The discrete adjoint problem \eqref{adj2} can be written as
\begin{equation}\label{dis_adj}
\tilde{\cL}_h(\Phi_h,\Theta_{\bar{u},h}):=A(\Phi_h,\Theta_{\bar{u},h})+2B(\Psi_{\bar{u},h},\Phi_h,\Theta_{\bar{u},h})=(\Psi_{\bar{u},h}-\Psi_d,\Phi_h)\quad\forall \Phi_h\in\bV_h.
\end{equation}
The subtraction of \eqref{dis_adj} and \eqref{dis_post_adj} leads to
\begin{equation}\label{sub_adj}
\tilde{\cL}_h(\Phi_h,\Theta_{\bar{u},h}-\Theta_{\cP_h\bar{u},h})=(\Psi_{\bar{u},h}-\Psi_{\cP_h\bar{u},h},\Phi_h).
\end{equation}
We consider a well-posed auxiliary problem:

Find $\boldsymbol{\chi}_h\in\bV_h$ such that
\begin{equation}\label{aux_adj}
\tilde{\cL}_h(\boldsymbol{\chi}_h,\Phi_h)=(\Theta_{\bar{u},h}-\Theta_{\cP_h\bar{u},h},\Phi_h)\quad\forall\Phi_h\in\bV_h
\end{equation}
with the $a~priori$ bound $\trinl\boldsymbol{\chi}_h\trinr_2\leq C\trinl\Theta_{\bar{u},h}-\Theta_{\cP_h\bar{u},h}\trinr$. Choose $\Phi_h=\boldsymbol{\chi}_h$ in \eqref{sub_adj} and $\Phi_h=\Theta_{\bar{u},h}-\Theta_{\cP_h\bar{u},h}$ in \eqref{aux_adj} to obtain
\begin{equation}
\trinl\Theta_{\bar{u},h}-\Theta_{\cP_h\bar{u},h}\trinr^2=(\Psi_{\bar{u},h}-\Psi_{\cP_h\bar{u},h},\boldsymbol{\chi}_h).
\end{equation}
The Cauchy-Schwarz inequality, \Poincare inequality, well-posedness of \eqref{aux_adj} and Theorem~\ref{PP_aux} lead to 
\begin{equation*}
\trinl\Theta_{\bar{u},h}-\Theta_{\cP_h\bar{u},h}\trinr\leq C\trinl\Psi_{\bar{u},h}-\Psi_{\cP_h\bar{u},h}\trinr\leq Ch^2\|\bar{u}\|_{\bar{h}}.
\end{equation*}
This completes the proof.
\end{proof}

{Choose the load function in \eqref{dis_adj} as $\Psi_{\bar{u}_h,h}-\Psi_{d}$, proceed as in the proof of Lemma~\ref{adj_post} and use Corollary~\ref{uh_control} to obtain the next result.}
\begin{cor}\label{adj_post_cor}
Let $\bar\Theta_h$ be solution of \eqref{adj2} with respect to control $\bar{u}_h$ and $\Theta_{\cP_h\bar{u},h}$ be the solution of \eqref{dis_post_adj}. Then the following error estimate holds true:
\begin{equation*}
\trinl \bar\Theta_h-\Theta_{\cP_h\bar{u},h}\trinr\leq C h^2\|\bar{u}\|_{\bar{h}}.
\end{equation*}
\end{cor}

\begin{lem}[A variational inequality] \cite[(3.15)]{MeyerRosh04}  The post processed control $\cP_h\bar{u}$ satisfies the variational inequality
	\begin{equation}\label{var_inequality}
	\|\cP_h\bar{u}- \bar{u}_h\|_{L^2(\omega)}^2 \le C(\cP_h \bar\Theta - \bar\Theta_h, \bar{u}_h - \cP_h\bar{u}).
	\end{equation}
\end{lem}

The proof of the next lemma is standard (for example \cite{GudiNatarajPorwal14,MeyerRosh04}). However, we provide a proof for the sake of completeness.
\begin{thm}[Convergence rate at centroids]\label{conv_centroid} Under the assumption {\bf (A3)}, the estimate
\begin{equation*}
\|\bar{u}_h-\cP_h\bar{u}\|_{L^2(\omega)}\leq Ch^\beta
\end{equation*}	
holds true with $\beta=\min\{2\gamma,2\},\,\gamma\in (1/2,1]$ being the index of elliptic regularity.
\end{thm}

\begin{proof}
A use of \eqref{var_inequality} and simple manipulations lead to
\begin{align}
&\|\cP_h\bar{u}- \bar{u}_h\|_{L^2(\omega)}^2 \lesssim (\cP_h \bar\Theta - \bar\Theta_h, \bar{u}_h - \cP_h\bar{u})\notag\\
&=(\cP_h \bar\Theta - \bar\Theta, \bar{u}_h - \cP_h\bar{u})+(\bar\Theta - \Theta_{\cP_h\bar{u},h}, \bar{u}_h - \cP_h\bar{u})+(\bar\Theta_{\cP_h\bar{u},h} - \bar\Theta_h, \bar{u}_h - \cP_h\bar{u}).\label{cent_split}
\end{align}
The first term is estimated using the fact that $\bar{u}_h-\cP_h\bar{u}$ is a constant in each $T \in \cT_h $ and hence,
\begin{equation*}
(\cP_h \bar\Theta - \bar\Theta, \bar{u}_h - \cP_h\bar{u})=\sum_{T_i\in\cT} (\bar{u}(S_i)-\cP_h\bar{u}(S_i))\int_{T_i}(\cP_h \bar\Theta - \bar\Theta)\dx.
\end{equation*}
Since $\bar\Theta|_{T_i}\in H^2(T_i)$, a use of \eqref{ni} in the above equation and $a~priori$ bound of $\bar\Theta$ from \eqref{state_eq}-\eqref{adj_eq} as $\trinl\bar\Theta\trinr_2\leq C(\|f\|+\|\bar{u}\|+\|y_d\|)$ lead to
\begin{equation}\label{adj_est1}
(\cP_h \bar\Theta - \bar\Theta, \bar{u}_h - \cP_h\bar{u})\leq Ch^2\|\bar{u}_h-\cP_h\bar{u}\|_{L^2(\omega)} \trinl\bar\Theta\trinr_2.
\end{equation}
The triangle inequality, Lemma~\ref{adj_post}, \Poincare inequality and \eqref{adjest} yield
\begin{equation}\label{Theta_est}
\trinl\bar\Theta-\Theta_{\cP_h\bar{u},h}\trinr\leq Ch^\beta
\end{equation}
with $\beta=\min\{2\gamma,2\}$. The equation \eqref{Theta_est}, and Cauchy-Schwarz inequality leads to the estimate for the second term of \eqref{cent_split} as
\begin{equation}\label{adj_est2}
|(\bar\Theta - \Theta_{\cP_h\bar{u},h}, \bar{u}_h - \cP_h\bar{u})|\leq Ch^\beta\|\bar{u}_h-\cP_h\bar{u}\|_{L^2(\omega)}.
\end{equation}
The Cauchy-Schwarz inequality and Corollary~\ref{adj_post_cor} lead to the estimate for the last term of \eqref{cent_split} as
\begin{equation}\label{adj_est3}
(\bar\Theta_{\cP_h\bar{u},h} - \bar\Theta_h, \bar{u}_h - \cP_h\bar{u})\leq Ch^\beta\|\bar{u}_h-\cP_h\bar{u}\|_{L^2(\omega)}.
\end{equation}
A use of the estimates \eqref{adj_est1}-\eqref{adj_est3} in \eqref{cent_split} leads to the required estimate.
\end{proof}

\begin{thm}[Estimate for post-processed control]\label{est_post_proc_u} The following estimate for post-processed control holds true:
\begin{equation*}
\|\bar{u}-\tilde{\bar{u}}_h\|_{L^2(\omega)}\leq Ch^\beta,
\end{equation*}	
where $\bar{u}$ is the optimal control and $\tilde{\bar{u}}_h$ is the post-processed control defined in \eqref{defn_post_proc}, and $\beta=\min\{2\gamma,2\}$. 
\end{thm}
\begin{proof}
The Lipschitz continuity of the projection operator $\pi_{[u_a,u_b]}$ and triangle inequality yield
\begin{align*}
\|\bar{u}-\tilde{\bar{u}}_h\|_{L^2(\omega)}&\leq \|\pi_{[u_a,u_b]} \big{(}-\frac{1}{\alpha}\bar\Theta\big{)}-\pi_{[u_a,u_b]}\big{(}-\frac{1}{\alpha}\bar\Theta_h\big{)}\|\leq C\trinl\bar\Theta-\bar\Theta_h\trinr\\
&\leq C(\trinl\bar\Theta-\Theta_{\cP_h\bar{u},h}\trinr+\trinl\Theta_{\cP_h\bar{u},h}-\bar\Theta_h\trinr).
\end{align*}
Now \eqref{Theta_est} and Corollary~\ref{adj_post_cor} lead to the required result.
\end{proof}


\begin{thm}\label{PP_aux_apdex}
Let $\Psi_{\bar{u},h}$ and $\Psi_{\cP_h\bar{u},h}$ be solution of \eqref{wform2} with respect to control $\bar{u}$ and post processed control $\cP_h\bar{u}$, respectively. Then the following error estimate holds true:
\begin{equation*}
\trinl \Psi_{\bar{u},h}-\Psi_{\cP_h\bar{u},h}\trinr_1\leq C h^2\|\bar{u}\|_{\bar{h}}.
\end{equation*}
\end{thm}  

\begin{proof}
Consider the perturbed auxiliary problem:
Seek $\bxi\in\bV$ that solves
\begin{equation}\label{pert_aux_apdex}
A(\bz,\bxi)+B(\Psi_{\bar{u},h},\bz,\bxi)+B(\bz,\Psi_{\cP_h\bar{u},h},\bxi)=-(\Delta(\Psi_{\bar{u},h}-\Psi_{\cP_h\bar{u},h}),\bz)\quad\forall \bz\in\bV.
\end{equation}
Its discretization is given by: 

Seek $\bxi_h\in\bV_h$ that solves
\begin{equation}\label{pert_aux_apdex_dis}
A(\bz_h,\bxi_h)+B(\Psi_{\bar{u},h},\bz_h,\bxi_h)+B(\bz_h,\Psi_{\cP_h\bar{u},h},\bxi_h)=-(\Delta(\Psi_{\bar{u},h}-\Psi_{\cP_h\bar{u},h}),\bz_h)\quad\forall \bz_h\in\bV_h.
\end{equation}
The above equation \eqref{pert_aux_apdex_dis} can be written in the operator form as
\begin{equation}
\cA^*\bxi_h+\cB'(\Psi_{\bar{u},h})^*\bxi_h+\half\cB'(\Psi_{\cP_h\bar{u},h}-\Psi_{\bar{u},h})^*\bxi_h=-T_h(\Delta(\Psi_{\bar{u},h}-\Psi_{\cP_h\bar{u},h}))\text{ in }\bV_h.
\end{equation}
Note that \eqref{apriori_state} and \eqref{scale_1} lead to 
\begin{equation}
\trinl\Psi_{\cP_h\bar{u},h}-\Psi_{\bar{u},h}\trinr_2\leq C \|\cP_h\bar{u}-\bar{u}\|_{L^2(\omega)}\leq Ch. \label{temp10_apdex}
\end{equation}
The invertibility of $\cA^*+\cB'(\Psi_{\bar{u},h})^*$, Lemma~\ref{banach} and \eqref{temp10_apdex} lead to well-posedness of \eqref{pert_aux_apdex_dis}.
Choose $\bz_h=\Psi_{\bar{u},h}-\Psi_{\cP_h\bar{u},h}$ in \eqref{pert_aux_apdex_dis} and simplify the terms to obtain 
\begin{align}
\trinl\Psi_{\bar{u},h}-\Psi_{\cP_h\bar{u},h}\trinr^2_1&=A(\Psi_{\bar{u},h}-\Psi_{\cP_h\bar{u},h},\bxi_h)+B(\Psi_{\bar{u},h},\Psi_{\bar{u},h},\bxi_h)\notag\\
&\qquad-B(\Psi_{\cP_h\bar{u},h},\Psi_{\cP_h\bar{u},h},\bxi_h).\label{L2app_apdex}
\end{align}
Note that $\Psi_{\bar{u},h}$ and $\Psi_{\cP_h\bar{u},h}$ satisfy the following discrete problems:
\begin{align*}
&A(\Psi_{\bar{u},h},\Phi_h)+B(\Psi_{\bar{u},h},\Psi_{\bar{u},h},\Phi_h)=(F+\cC\bar{u},\Phi_h)\quad\forall \Phi_h\in \bV_h,\\
&A(\Psi_{\cP_h\bar{u},h},\Phi_h)+B(\Psi_{\cP_h\bar{u},h},\Psi_{\cP_h\bar{u},h},\Phi_h)=(F+\cC(\cP_h\bar{u}),\Phi_h)\quad\forall \Phi_h\in \bV_h.
\end{align*}
Subtract the above two equations to obtain
\begin{align*}
A(\Psi_{\bar{u},h}-\Psi_{\cP_h\bar{u},h},\Phi_h)+B(\Psi_{\bar{u},h},\Psi_{\bar{u},h},\Phi_h)-B(\Psi_{\cP_h\bar{u},h},\Psi_{\cP_h\bar{u},h},\Phi_h)=(\cC(\bar{u}-\cP_h\bar{u}),\Phi_h).
\end{align*}
Choose $\Phi_h=\bxi_h$ in the above equation and use \eqref{L2app_apdex} to obtain
\begin{equation}\label{pert_est_apdex}
\trinl\Psi_{\bar{u},h}-\Psi_{\cP_h\bar{u},h}\trinr^2_1=(\cC(\bar{u}-\cP_h\bar{u}),\bxi_h).
\end{equation} 
{
Now proceed as in the proof of Theorem~\ref{PP_aux} to obtain the required estimate.}
\end{proof}

Following the proof of the above theorem, the next result holds immediately.
\begin{cor}\label{uh_control_apdex}
Let $\Psi_{\bar{u}_h,h}$ and $\Psi_{\cP_h\bar{u},h}$ be the solutions of \eqref{wform2} with the control $\bar{u}_h$ and the post processed control $\cP_h\bar{u}$, respectively. Then the following error estimate holds true:
\begin{equation*}
\trinl \Psi_{\bar{u}_h,h}-\Psi_{\cP_h\bar{u},h}\trinr_1\leq C h^2\|\bar{u}\|_{\bar{h}}.
\end{equation*}
\end{cor}

The discrete post processed adjoint problem can be stated as: 

$\quad$Find $\Theta_{\cP_h\bar{u},h}\in\bV_h$ such that
\begin{equation}\label{dis_post_adj_apdex}
\tilde{\cL}_h(\Phi_h,\Theta_{\cP_h\bar{u},h}):=A(\Phi_h,\Theta_{\cP_h\bar{u},h})+2B(\Psi_{\bar{u},h},\Phi_h,\Theta_{\cP_h\bar{u},h})=(\Psi_{\cP_h\bar{u},h}-\Psi_d,\Phi_h)\quad\forall \Phi_h\in\bV_h.
\end{equation}

\begin{lem}\label{adj_post_apdex}
Let $\Theta_{\bar{u},h}$ be solution of \eqref{adj2} with the control $\bar{u}$ and $\Theta_{\cP_h\bar{u},h}$ be the solution of \eqref{dis_post_adj_apdex}. Then the following error estimate holds true:
\begin{equation*}
\trinl \Theta_{\bar{u},h}-\Theta_{\cP_h\bar{u},h}\trinr_1\leq C h^2\|\bar{u}\|_{\bar{h}}.
\end{equation*}
\end{lem}
\begin{proof}
The discrete adjoint problem \eqref{adj2} can be written as
\begin{equation}\label{dis_adj_apdex}
\tilde{\cL}_h(\Phi_h,\Theta_{\bar{u},h}):=A(\Phi_h,\Theta_{\bar{u},h})+2B(\Psi_{\bar{u},h},\Phi_h,\Theta_{\bar{u},h})=(\Psi_{\bar{u},h}-\Psi_d,\Phi_h)\quad\forall \Phi_h\in\bV_h.
\end{equation}
The subtraction of \eqref{dis_adj_apdex} and \eqref{dis_post_adj_apdex} leads to
\begin{equation}\label{sub_adj_apdex}
\tilde{\cL}_h(\Phi_h,\Theta_{\bar{u},h}-\Theta_{\cP_h\bar{u},h})=(\Psi_{\bar{u},h}-\Psi_{\cP_h\bar{u},h},\Phi_h).
\end{equation}
The proof follows exactly similar to that of Lemma~\ref{adj_post} except for the change that in place of \eqref{aux_adj}, we consider the following well-posed auxiliary problem:
Find $\boldsymbol{\chi}_h\in\bV_h$ such that
\begin{equation}\label{aux_adj_apdex}
\tilde{\cL}_h(\boldsymbol{\chi}_h,\Phi_h)=(-\Delta(\Theta_{\bar{u},h}-\Theta_{\cP_h\bar{u},h}),\Phi_h)\quad\forall\Phi_h\in\bV_h
\end{equation}
with the $a~priori$ bound $\trinl\boldsymbol{\chi}_h\trinr_2\leq C\trinl-\Delta(\Theta_{\bar{u},h}-\Theta_{\cP_h\bar{u},h})\trinr_{-1}\leq C\trinl \Theta_{\bar{u},h}-\Theta_{\cP_h\bar{u},h}\trinr_1$. 
\end{proof}

\begin{cor}\label{adj_post_apdex_cor}
Let $\bar\Theta_h$ be solution of \eqref{adj2} with respect to control $\bar{u}_h$ and $\Theta_{\cP_h\bar{u},h}$ be the solution of \eqref{dis_post_adj_apdex}. Then the following error estimate holds true:
\begin{equation*}
\trinl \bar\Theta_h-\Theta_{\cP_h\bar{u},h}\trinr_1\leq C h^2\|\bar{u}\|_{\bar{h}}.
\end{equation*}
\end{cor}

\begin{thm}[$H^1$ and $L^2$-estimates for state and adjoint variables]\label{L2_est_state_adj_apdex}
Let $({\bar \Psi}, {\bar u})\in \bV\times L^2(\omega)$ be a nonsingular solution of \eqref{of}.
Let $\bar{\Psi}$ and $\bar{\Psi}_h$ be solutions of \eqref{wform} and \eqref{discrete_cost} respectively, and $\bar \Theta$ and $\bar\Theta_h$ be the solutions of the corresponding adjoint problems. For sufficiently small $h$, the following estimates hold true:
\begin{align*}
\text{(a) } &\trinl {\bar \Psi} - {\bar \Psi}_h\trinr_1\leq Ch^{2\gamma},\quad \trinl {\bar \Theta}- {\bar \Theta}_h \trinr_1\leq Ch^{2\gamma}. \\
\text{(b) } &\trinl\bar{\Psi}-\bar{\Psi}_h\trinr\leq Ch^{2\gamma},\quad \,\; 
\trinl {\bar \Theta}-{\bar \Theta}_h \trinr\leq Ch^{2\gamma}.
\end{align*}
\end{thm}

\begin{proof}
{
The triangle inequality, Theorems~\ref{ee1} and~\ref{PP_aux_apdex} and Corollary~\ref{uh_control_apdex} lead to 
\begin{align*}
\trinl\bar{\Psi}-\bar{\Psi}_h\trinr_1\leq \trinl\bar{\Psi}-\Psi_{\bar{u},h}\trinr_1+\trinl\Psi_{\bar{u},h}-\Psi_{\cP_h\bar{u},h}\trinr_1+\trinl\Psi_{\cP_h\bar{u},h}-\bar{\Psi}_h\trinr_1\leq Ch^{2\gamma}.
\end{align*}
Similarly, the triangle inequality, Theorems~\ref{ee2} and ~\ref{adj_post_apdex} and Corollary~\ref{adj_post_apdex_cor} lead to 
\begin{align*}
\trinl\bar{\Theta}-\bar{\Theta}_h\trinr_1\leq \trinl\bar{\Theta}-\Theta_{\bar{u},h}\trinr_1+\trinl\Theta_{\bar{u},h}-\Theta_{\cP_h\bar{u},h}\trinr_1+\trinl\Theta_{\cP_h\bar{u},h}-\bar{\Theta}_h\trinr_1\leq Ch^{2\gamma}.
\end{align*}
This completes the proof of part (a). Part (b) follows easily. }
\end{proof}

\section{Numerical Results}

In this section, we present two numerical examples to illustrate the theoretical estimates obtained in this paper. The discrete optimization problem \eqref{discrete_cost} is solved using the primal-dual active set strategy \cite{fredi2010}. The state and adjoint variables are discretized using \BFS finite elements and the control variable is discretized using piecewise constants. Further, the post-processed control is computed with the help of the discrete adjoint variable.  Let the $l$-th level error and mesh parameter be denoted by $e_l$ and $h_l$, respectively. The $l$-th level experimental order of convergence is defined by
\begin{equation*}
    \delta_l:=log(e_{l}/e_{l-1})/log(h_{l}/h_{l-1}).
\end{equation*}
The errors and numerical orders of convergence are presented for both the examples.

\medskip
{\bf Example 1.} Let the computational domain be $\Omega=(0,1)^2$. A slightly modified version of \eqref{cost}-\eqref{state3} is constructed in such a way that that the exact solution is known. This is done by choosing the state variables $\bar{\psi}_1,\bar{\psi}_2$ and the adjoint variables $\bar{\theta}_1,\bar{\theta}_2$ as
\begin{equation*}
\bar{\psi}_1=\bar{\psi}_2=\sin^2(\pi x)\sin^2(\pi y),\quad \bar{\theta}_1=\bar{\theta}_2=x^2 y^2 (1-x)^2 (1-y)^2,
\end{equation*}
and the control $\bar{u}$ as $\bar{u}(x)=\pi_{[-750,-50]}\left(-\frac{1}{\alpha}\bar{\theta}_1(x)\right)$, where the regularization parameter $\alpha$ is chosen as $10^{-5}$.

The source terms $f,\tilde{f}$ and observation $\bar{\Psi}_d=(\bar{\psi}_{1d},\bar{\psi}_{2d})$ are then computed using
\begin{align*}
&f=\Delta^2\bar{\psi}_1-[\bar{\psi}_1,\bar{\psi}_2]- \bar{u},\quad \tilde{f}=\Delta^2\bar{\psi}_2+\half[\bar{\psi}_1,\bar{\psi}_1]\quad\text{ and }\\
& \bar{\psi}_{1d}=\bar{\psi}_1-\Delta^2\bar{\theta}_1,\quad \bar{\psi}_{2d}=\bar{\psi}_2-\Delta^2\bar{\theta}_2+[\bar{\psi}_1,\bar{\theta}_1].
\end{align*}
The errors and orders of convergence for the numerical approximations to state, adjoint and control variables are shown in Tables \ref{table:OC_Control_unsq}-\ref{table:OC_State_unsq}. 
 
 \smallskip
In all the tables, $h_0=1/\sqrt{2}$ is the initial mesh size and $N$ denotes the number of degrees of freedom. Since $\Omega$ is convex, we have the index of elliptic regularity $\gamma=1$. The numerical convergence rates with respect $H^1$ and $L^2$ norms for the state and adjoint variables are quadratic as predicted theoretically. Linear orders of convergence for the control variable and quadratic order of convergence for the post-processed control are obtained and this confirms the theoretical results established in Theorem~\ref{est_post_proc_u}. 
   
{\small{
\begin{table}[h!]
	\begin{center}
		\begin{tabular}{ | c c |c c|c c| c  c | c c|}     
			\hline
$ N $   & $h/h_0$   &$\trinl\bar{\Psi}-\bar{\Psi}_h\trinr_{2}$ & $\delta_l$   &$\trinl\bar\Theta-\bar\Theta_h\trinr_{2}$ & $\delta_l$   &$\|\bar{u}-\bar{u}_{h}\|$ & $\delta_l$ & $\|\bar{u}-\tilde{\bar{u}}_{h}\|$ & $\delta_l$  \\ 
\hline
36     &$2^{-1}$ & 1.60389465 & -     &0.00479710 & -     & 46.72538744 & -     & 0.68424095  & - \\ 
196    &$2^{-2}$ & 0.41295628 & 1.957 &0.00121897 &1.976  & 25.52270587 & 0.872 & 0.37526702  & 0.866\\ 

900    &$2^{-3}$ & 0.10369078 & 1.993 &0.00030420 &2.002  & 12.92074925 & 0.982 & 0.11011474  & 1.768\\ 

3844   &$2^{-4}$ & 0.02592309 & 1.999 &0.00007602 &2.000  &  6.53425879 & 0.983 & 0.02846417  & 1.951\\ 

15876  &$2^{-5}$ & 0.00648563 & 1.998 &0.00001900 &2.000  &  3.27120390 & 0.998 & 0.00717641  & 1.987  \\ 

64516  &$2^{-6}$ & 0.00161877 & 2.002 &0.00000475 &1.999  &  1.63710571 & 0.998 & 0.00179764  & 1.997  \\ 
\hline
\end{tabular}
\end{center}
\caption{Errors and orders of convergence for the state, adjoint, control and post processed control variables in {\bf Example 1}}
\label{table:OC_Control_unsq}
\end{table}}}

{\small{
\begin{table}[h!]
	\begin{center}
		\begin{tabular}{ | c c  |c c | c c |c c |c c|}     
			\hline			
$ N $   & $h/h_0$        & $\trinl\bar{\Psi}-\bar{\Psi}_h\trinr_{1}$ & $\delta_l$ & $\trinl\bar{\Psi}-\bar{\Psi}_h\trinr$ & $\delta_l$ & $\trinl\bar\Theta-\bar\Theta_h\trinr_{1}$ & $\delta_l$ &$\trinl\bar\Theta-\bar\Theta_h\trinr$& $\delta_l$\\
\hline 
36   &$2^{-1}$  & 0.06403252 & -     &0.80432845E-2 & -    &0.24927604E-3 & -    &0.37036191E-4& -  \\ 
196  &$2^{-2}$  & 0.01290311 & 2.311 &0.17586110E-2 &2.193 &0.05657690E-3 &2.139 &0.08441657E-4& 2.133    \\ 

900  &$2^{-3}$  & 0.00315079 & 2.033 &0.04899818E-2 &1.843 &0.01386737E-3 &2.028 &0.02134139E-4& 1.983   \\ 

3844 &$2^{-4}$  & 0.00077657 & 2.020 &0.01226971E-2 &1.997 &0.00345918E-3 &2.003 &0.00537248E-4& 1.989  \\ 

15876&$2^{-5}$  & 0.00019514 & 1.992 &0.00312879E-2 &1.971 &0.00086273E-3 &2.003 &0.00134258E-4&  2.000 \\ 

64516&$2^{-6}$  & 0.00004781 & 2.029 &0.00075095E-2 &2.058 &0.00021657E-3 &1.994 &0.00033748E-4&  1.992 \\
\hline 
      \end{tabular}
	\end{center}
	\caption{$H^1$ and $L^2$ errors and orders of convergence for the state and adjoint variables in {\bf Example 1}}
	\label{table:OC_State_unsq}
\end{table}}}

{\bf Example 2.} Let $\Omega$ be the non-convex L-shaped domain $\Omega=(-1,1)^2 \setminus\big{(}[0,1)\times(-1,0]\big{)}$. We consider a problem with the exact singular solution borrowed from~\cite{Grisvard} in polar coordinates. The state and adjoint variables $\bar{\Psi}=(\bar{\psi}_{1},\bar{\psi}_2)$ and $\bar\Theta=(\bar{\theta}_1,\bar{\theta}_2)$ are given by
	\begin{align*}
	\bar{\psi}_1=\bar{\psi}_2=\bar{\theta}_1=\bar{\theta}_2=(r^2 \cos^2\theta-1)^2 (r^2 \sin^2\theta-1)^2 r^{1+ \gamma}g_{\gamma,\omega}(\theta)
	\end{align*}
	where $\omega=\frac{3\pi}{2}$, and $ \gamma\approx 0.5444837367$ is a non-characteristic 
	root of $\sin^2( \gamma\omega) =  \gamma^2\sin^2(\omega)$ with
	\begin{align*}
	g_{ \gamma,\omega}(\theta)&=\left(\frac{1}{ \gamma-1}\sin\big{(}( \gamma-1)\omega\big{)}-\frac{1}{ \gamma+1}\sin\big{(}( \gamma+1)\omega\big{)}\right)\times\Big{(}\cos\big{(}( \gamma-1)\theta\big{)}-\cos\big{(}( \gamma+1)\theta\big{)}\Big{)}\\
	&\quad-\left(\frac{1}{ \gamma-1}\sin\big{(}( \gamma-1)\theta\big{)}-\frac{1}{ \gamma+1}\sin\big{(}( \gamma+1)\theta\big{)}\right)\times\Big{(}\cos\big{(}( \gamma-1)\omega\big{)}-\cos\big{(}( \gamma+1)\omega\big{)}\Big{)}.
	\end{align*}  
The exact control $\bar{u}$ is chosen as $\bar{u}(x)=\pi_{[-600,-50]} \left( -
\frac{1}{ \alpha}\bar{\theta}_1(x) \right)$, where $\alpha=10^{-3}$. The source terms $f,\tilde{f}$ and the observation $\Psi_d=(\psi_{1d},\psi_{2d})$ are computed as in the previous example.
The errors and orders of convergence for the numerical approximations to state, adjoint and control variables are shown in Tables~\ref{table:OC_Control_LS}-\ref{table:OC_State_LS}. Since $\Omega$ is non-convex, we expect only $1/2<\gamma <1$ as predicted by the theoretical results. Note that only suboptimal orders of convergence are attained for the state and adjoint variables in the energy,  $H^1$ and $L^2$ norms.  However, we observe a linear order of convergence for the control variable and $2\gamma$ rate of convergence for the post-processed control and this confirms the theoretical results established in  Theorem~\ref{est_post_proc_u}.
{\small{
\begin{table}[h!]
	\begin{center}
		\begin{tabular}{ | c c |c c | c c  |c  c | c c|}     
			\hline
$ N $   & $h/h_0$   &$\trinl\bar{\Psi}-\bar{\Psi}_h\trinr_{2}$ & $\delta_l$  &$\trinl\bar\Theta-\bar\Theta_h\trinr_{2}$ & $\delta_l$    & $\|\bar{u}-\bar{u}_{h}\|$ & $\delta_l$ & $\|\bar{u}-\tilde{\bar{u}}_{h}\|$ & $\delta_l$  \\ 
\hline

36     &$2^{-1}$ &9.81990941 &   -   &7.47714482 &-    & 242.759537  & -      &34.2926324  & - \\ 

164    &$2^{-2}$ &2.95442143 & 1.732 &2.82045689 &1.406 & 116.204133 & 1.062 & 9.72134519  & 1.818\\ 

708    &$2^{-3}$ &1.41082575 & 1.066 &1.35893052 &1.053 & 61.137057  & 0.926 & 5.29868866  & 0.875\\ 

2948   &$2^{-4}$ &0.82993102 & 0.765 &0.82022205 &0.728 & 31.226881  & 0.969 & 2.54226401  & 1.059\\ 

12036  &$2^{-5}$ &0.54373393 & 0.610 &0.54214544 &0.597 & 15.691309 & 0.992 & 1.17813556  & 1.109  \\ 

48644  &$2^{-6}$ &0.36837935 & 0.561 &0.36796971 &0.559 & 7.860646  & 0.997 & 0.55275176  & 1.091  \\ 
\hline
\end{tabular}
\end{center}
\caption{Errors and orders of convergence for the state, adjoint, control and post-processed control variables in {\bf Example 2}}
\label{table:OC_Control_LS}
\end{table}}}

          {\small{
\begin{table}[h!]
	\begin{center}
		\begin{tabular}{ | c c |c  c  | c c | c c| c c|}     
			\hline			
$ N $   & $h/h_0$ &  $\trinl\bar{\Psi}-\bar{\Psi}_h\trinr_{1}$ & $\delta_l$ & $\trinl\bar{\Psi}-\bar{\Psi}_h\trinr$ & $\delta_l$ & $\trinl\bar\Theta-\bar\Theta_h\trinr_{1}$ & $\delta_l$ &$\trinl\bar\Theta-\bar\Theta_h\trinr$& $\delta_l$\\ 
\hline 

36   &$2^{-1}$  & 1.10962279 & -     &0.26602624 & -    &0.46789881 & -    &0.05277374& - \\ 
164  &$2^{-2}$  & 0.15147063 & 2.872 &0.02879033 &3.207 &0.11399788 &2.037 &0.01813025& 1.5414   \\ 

708  &$2^{-3}$  & 0.06196779 & 1.289 &0.01416231 &1.023 &0.04533146 &1.330 &0.00982844& 0.883  \\ 

2948 &$2^{-4}$  & 0.02244196 & 1.465 &0.00484927 &1.546 &0.02080671 &1.123 &0.00479517& 1.035 \\ 

12036&$2^{-5}$  & 0.00895880 & 1.324 &0.00178338 &1.443 &0.00970478 &1.100 &0.00226287& 1.083 \\ 

48644&$2^{-6}$  & 0.00405435 & 1.143 &0.00080248 &1.152 &0.00455562 &1.091 &0.00106533& 1.086\\

\hline
      \end{tabular}
	\end{center}
	\caption{$H^1$ and $L^2$ errors and orders of convergence for the state and adjoint variables in {\bf Example 2}}
	\label{table:OC_State_LS}
\end{table}}}

\bigskip

\noindent {\Large \bf Conclusions}

\medskip

\noindent In this paper, an attempt has been made to establish error estimates for state, adjoint and control variables for distributed optimal control problems governed by the \vket defined over polygonal domains. The convergence results in energy, $H^1$ and $L^2$ norms for state and adjoint variables are derived under realistic regularity assumptions on the exact solution of the problem. Also, the convergence results in $L^2$ norm for the control variable and a post processed control are established. The results of the numerical experiments confirm the theoretical error estimates. The extension of the analysis to nonconforming finite element methods, say piecewise quadratic Morley finite element is quite attractive from the implementation perspective. However, for the control problem, the nonconformity of the Morley finite element space offers a lot of challenges in the extension of the theoretical error estimates. We are currently working on this problem.

\bigskip

\noindent {\Large \bf Acknowledgements} The authors are members of the Indo-French Centre for Applied Mathematics, UMI-IFCAM, Bangalore, India, 
supported by DST-IISc-CNRS and Universit\'e Paul Sabatier-Toulouse III. This work was carried out within the IFCAM-project `PDE-Control'.

\medskip
\bibliographystyle{amsplain}
\bibliography{vKeBib}

\end{document}